\documentclass[a4paper, 11pt, final,underruledhead,leqno]{paper}
\usepackage[mathscr]{eucal}
\usepackage{enumerate, amssymb, math, mathtools}

\usepackage[usenames,dvipsnames]{color}

\title{Affinization of Segre products of partial linear spaces}
\author{Krzysztof Petelczyc, Mariusz {\.Z}ynel}

\DeclareMathOperator{\GF}{GF}
\DeclareMathOperator{\GL}{GL}
\DeclareMathOperator{\covering}{SC}

\DeclareMathOperator{\colin}{\mathbf{L}}
\def\bigtimes{\mbox{\Large$\times$}}
\def\LineOn(#1,#2){\overline{{#1},{#2}\rule{0em}{1,5ex}}}
\def\adjac{\mathrel{\sim}}
\def\penc{{\bf p}}

\def\lines{{\cal L}}

\def\ProjectiveSpSymb{\mathbf{P}}
\def\ProjSpace(#1){\ensuremath{\ProjectiveSpSymb(#1)}}
\def\PencSpace(#1,#2){\ensuremath{\ProjectiveSpSymb_{#1}(#2)}}
\def\TopSpace(#1,#2){\ensuremath{\ProjectiveSpSymb^{\mbox{\tiny\upshape top}}_{#1}(#2)}}
\def\StarSpace(#1,#2){\ensuremath{\ProjectiveSpSymb^{\mbox{\tiny\upshape star}}_{#1}(#2)}}
\def\PencSpacex(#1,#2){\ensuremath{\ProjectiveSpSymb^{\ast}_{#1}(#2)}}
\def\StarOf(#1){{\mathrm{S}}({#1})}

\def\TopOf(#1){{\mathrm{T}}({#1})}

\def\AfSpace(#1){{\bf A}({#1})}
\def\AfSpaceg(#1,#2){{\bf A}_{{#1}}({#2})}
\def\hipa{\mathscr{H}}
\def\hipy{\mathcal{H}}

\def\prodparal{\mathrel{\parallel^\sim}}
\def\veblparal{\mathrel{\parallel^\circ}}
\def\astparal{\mathrel{\parallel^\ast}}
\def\quadrparal{\mathrel{\parallel_\circ}}

\def\kolczaty{spiky}
\def\luskwiaty{flappy}

\def\A{\mathfrak{A}}
\def\M{\mathfrak{M}}
\def\PS{\mathfrak{P}}

\def\subst(#1,#2,#3){{#1}[{#2}/{#3}]}
\def\bsubst(#1,#2,#3){{#1}\Bigl[{#2}/{#3}\Bigr]}

\def\Aaff{\A_{\text{aff}}}
\def\Aproj{\A_{\text{proj}}}

\newenvironment{cmath}{%
  \par
  \smallskip
  \centering
  $
}{%
  $
  \par
  \smallskip
  \csname @endpetrue\endcsname
}

\newenvironment{ctext}{%
  \par
  \smallskip
  \centering
}{%
 \par
 \smallskip
 \csname @endpetrue\endcsname
}

\def\myend{\hfill\mbox{\small$\bigcirc$}}

\begin{document}

\def\interpretname{Interpretation}
\newtheorem{intrprx}[thm]{\interpretname}
\newenvironment{intrpr}{\begin{intrprx}\normalfont}{\myend\end{intrprx}}

\maketitle

\begin{abstract}
  Hyperplanes and hyperplane complements in the Segre product of partial linear spaces
  are investigated . The parallelism of such a complement is characterized in terms of
  the point-line incidence. Assumptions, under which the automorphisms of the complement
  are the restrictions of the automorphisms of the ambient space, are given.
  An affine covering for the Segre product of Veblenian gamma spaces is established.
  A general construction that produces non-degenerate hyperplanes in the Segre 
  product of partial linear spaces embeddable into projective space is introduced.
\end{abstract}

\begin{flushleft}\small
\end{flushleft}



\section*{Introduction}

The term \emph{affinization} is not widely used. Its idea however, is
not only well known but also applied very often in geometry. It has been spotted
in \cite{afin:pasin} and means construction of the complement of a hyperplane
in some point-line space, inspired by construction of an affine space as a reduct
of a projective space. 
To be fair we should cite a lot more papers here. Those of a great impact
for our work are \cite{cohenshult} and \cite{cuypers}.
A problem that is closely related to the removal of a point subset or
a line subset or both is reconstruction of the ambient space from the remainder.
This is addressed in \cite{wielohol} for projective spaces 
and for Grassmann spaces, while \cite{segrel}
deals with the Segre product of Grassmann spaces.

The main part of the paper starts with the characterization of hyperplanes in the Segre product $\M$ of
partial linear spaces (Theorem \ref{thm:hip:inprod}). These are similar to
structures investigated in \cite{bertram1} and \cite{bertram2}.  Generalized
projective geometries  introduced in \cite{bertram1} are products of two
geometries, such that some distinguished subsets of this product have the
structure of an affine space. In the case of projective spaces these subsets are
directly affine spaces that emerge as hyperplane complements.  Likewise, we are
interested in locally affine structures obtained as a hyperplane complement
$\M\setminus\hipa$, the result of affinization.  In this context issues typical
to affine geometry with parallelism arise. Our goal is to solve some of them.  

\newpage

The automorphism group of $\M\setminus\hipa$ is characterized
(Theorem \ref{thm:auty:prod}).
We prove that the parallelism
$\parallel_\hipa$ is definable in terms of point-line incidence of the product
$\M$ (Proposition \ref{prop:parallelglobal}) like it is in most of 
geometries that resemble affine spaces. One of the exceptions could be spine spaces 
with affine lines only (cf. \cite{jez4jez}).

The next problem concerns the existence of a hyperplane $\hipa$ in $\M$ such that the
complement $\M\setminus\hipa$ is not isomorphic to the Segre product of the related hyperplane
complements taken in the components of $\M$. 
Those isomorphic to such products are relatively easy to find (cf. Proposition \ref{prop:isomorph}). 
In any case, under assumption that $\M$ is the product of Veblenian gamma spaces with 
lines thick enough  the complement $\M\setminus\hipa$ is covered by
affine spaces (Fact \ref{fact:covering}).

In the last section we focus on the Segre product $\M$ which components are
embeddable into projective space. For such $\M$ we introduce a general 
construction of a hyperplane, which idea is based on the characterisation of 
hyperplanes in Grassmann spaces provided in \cite{shult} (see also
\cite{debruyn}, \cite{hall}, \cite{hallshult}). This makes possible to
show that numerous non-degenerate hyperplanes in $\M$ do exist. 
The complete characterization of hyperplanes in $\M$ is
challenging and worth to be done, but it is not the goal of this
paper. 
This characterization frequently involves computations related to multilinear forms and 
hyperdeterminants. Many results in this area can be found in the literature 
or the Internet, but none of them gives an ultimate answer to our problems.
So, we can only explicitly characterize hyperplanes in the Segre product
of projective spaces, which is with no doubt the most significant class.


\section{Generalities}


\subsection{Partial linear spaces, hyperplanes and parallelism}

A structure $\M = \struct{S,\lines}$, $\lines\subseteq\sub(S)$, where
the elements of $S$ are called \emph{points} and the elements of $\lines$ 
are called \emph{lines}, is a \emph{partial linear space} iff
there are two or more points on every line,
there is a line through every point,
and any two lines that share two or more points coincide.
We say that the points $a, b\in S$ are \emph{collinear} or \emph{adjacent} in $\M$
and write $a\adjac b$ when they are on a line of $\M$. 
The set of all the points adjacent to a given point $a$
is $[a]_{\mathord{\adjac}} := \set{b\in S\colon a\adjac b}$.
A partial linear space where every two points are collinear is a \emph{linear space}.
Two lines $L, K$ are said to be \emph{adjacent}, in symbols
$L\adjac K$, whenever they share a point.

We say that three pairwise distinct points $a, b, c\in S$
form a \emph{triangle} in $\M$ if they are pairwise adjacent and not collinear.
A \emph{subspace} of $\M$ is a subset $X$ of
$S$ with the property that if a line $L$ shares two or more points with $X$, then
$L$ is entirely contained in $X$.
A subspace $X$ of $\M$ is \emph{strong} if any two points in $X$ are
collinear. 
We call a subspace $X$ of $\M$ a \emph{hyperplane} if it is proper and
every line of $\M$ meets $X$. In other words a hyperplane is a set $X$ of points
such that every line meets $X$ in one or all points.

A partial linear space is \emph{connected} iff adjacency relation $\adjac$ is connected i.e.\ when
any two points $p,q$ can be joined by a sequence
$p = a_0 \adjac a_1 \adjac\dots\adjac a_n = q$. 
It is \emph{strongly connected} iff given at least 2-element strong subspace $X$ 
and a point $p$, there is a sequence of strong subspaces
$Y_0,\dots,Y_n$ such that $X = Y_0$, $p\in Y_n$, and $2 \le\abs{Y_{i-1}\cap Y_i}$ 
for all $i=1,\dots,n$. It is clear that every strongly connected
partial linear space is connected. 
In what follows we restrict ourselves to connected partial linear spaces.

Our results involve two specific properties of hyperplanes which we define here in 
general setting. A subset $X$ of $S$ is called 
\begin{itemize}\itemsep-2pt
  \item
  \emph{spiky} when every point $a\in X$ is adjacent to
  some point $b\notin X$,
  \item
  \emph{\luskwiaty} when for every line $L\subseteq X$ there is a point
  $a\notin X$ such that $L\subseteq [a]_{\mathord{\adjac}}$.
\end{itemize}

\begin{lem}\label{prop:wystaje2kolczaty}
  A \luskwiaty\ hyperplane of a partial linear space is \kolczaty.
\end{lem}

\begin{proof}
  Let $\M = \struct{S,\lines}$ be a partial linear space, let
  $\hipa$ be a hyperplane in $\M$. Suppose that $\hipa$  is not \kolczaty.
  Then there is a point $q\in \hipa$ such that each line through $q$ is entirely
  contained in $\hipa$. Let $q\in L\in\lines$. As $\hipa$ is \luskwiaty\ there
  is $a\notin \hipa$ such that $L\subseteq [a]_{\mathord{\adjac}}$.
  In particular, $a \adjac q$, which is impossible.
\end{proof}
However, a \kolczaty\ hyperplane need not to be \luskwiaty. 
\begin{exm}\label{exm:spiky:nonfloppy}
Let $\PS=\struct{S,\lines}$ be a projective $3$-space, $L$ be a line, and $\hipa$ a 
hyperplane in $\PS$ such that $L\subseteq\hipa$. Take a line $M$, which is skew to $L$. 
There is a bijection $f\colon L\longrightarrow M$. 
Let $\lines_1=\bset{\LineOn(a,{f(a)})\colon a\in L}$ and $\lines_2=\set{K\in\lines\colon K\cap L=\emptyset}$. 
Then $\hipa$ is \kolczaty\ but non-\luskwiaty\ in $\struct{S,\lines_1\cup\lines_2}$.
\end{exm}
%
%


Obviously, in case of linear spaces all hyperplanes are \luskwiaty\, and thus \kolczaty\ as well.
Moreover, in this case hyperplanes are maximal proper subspaces, so
there are no distinct hyperplanes such that one is contained in the other.
However, it is possible in partial linear spaces.

\begin{exm}
  Take a projective space $\PS = \struct{S, \lines}$ that is at least a plane.
  Let $X_1, X_2$ be two distinct hyperplanes in $\PS$.
  Set $\hipa_1 := X_1\cap X_2$,\ $\hipa_2 := X_2$, and
  \begin{cmath}
    \M := \bstruct{X_1\cup X_2,
      { \{ L\in\lines\colon L\subseteq X_1\cup X_2 \} }}.
  \end{cmath}
  It is clear that $\M$ is a partial linear space where $\hipa_1, \hipa_2$ are two
  distinct hyperplanes with $\hipa_1\subsetneq\hipa_2$.
\end{exm}

One can also say that hyperplanes in linear spaces are minimal sets satisfying
their definition. Spiky hyperplanes, which are of our principal concern in this
paper, exhibit similar behaviour in partial linear spaces, they are minimal
sets.

\begin{lem}\label{lem:kolczat2minim}
  Let\/ $\hipa_1, \hipa_2$ be hyperplanes in a partial linear space 
  with $\hipa_1\subseteq\hipa_2$. If\/ $\hipa_2$ is \kolczaty, then 
  $\hipa_1 = \hipa_2$.
\end{lem}

\begin{proof}
  Suppose that there is a point $p\in\hipa_2\setminus\hipa_1$. As $\hipa_2$ is
  spiky,  there is a  line $L$ such that $L\cap\hipa_2 = \set{p}$. But
  then $L\cap\hipa_1 = \emptyset$, a contradiction.
\end{proof}

A hyperplane restricted to a substructure is a hyperplane in that substructure.

\begin{lem}\label{lem:hip:restricted}
  Let\/ $S_0\subseteq S$, $\lines_0\subseteq\lines\cap\sub(S_0)$, 
  and $\hipa$ be a hyperplane in\/ $\M$.
  If\/ $S_0\nsubseteq\hipa$, then $\hipa\cap S_0$ is a hyperplane in $\struct{S_0,\lines_0}$.
\end{lem}

\begin{proof}
  Set $\M_0 := \struct{S_0,\lines_0}$ and $\hipa_0 := \hipa\cap S_0$.
  Let $L\in\lines_0$ be such that $\abs{L\cap \hipa_0}\geq 2$. Then 
  $\abs{L\cap \hipa}\geq 2$. Since $L\subseteq S_0$, we have $L\subseteq\hipa_0$.
  Therefore $\hipa_0$ is a subspace of $\M_0$.
  The assumption  $S_0\nsubseteq\hipa$ implies that $\hipa_0$ is a proper subspace of $\M_0$.
  Next, let $K\in\lines_0$. Notice that there is $p\in K\cap\hipa$.
  As $K\subseteq S_0$ we get $p\in K\cap\hipa_0$ which completes the proof.
\end{proof}

A \emph{gamma space} is a partial linear space where $[a]_{\mathord{\adjac}}$ 
is a subspace for all $a\in S$. Gamma spaces are also known as those partial
linear spaces satisfying \emph{none-one-or-all} axiom.
A partial linear space is said to be \emph{Veblenian} iff
for any two distinct lines $L_1, L_2$ through a point $p$ and
any two distinct lines $K_1, K_2$ not through the point $p$
whenever $L_1,L_2 \adjac K_1, K_2$, then $K_1 \adjac K_2$.
Note that a projective space is a Veblenian linear space with lines 
of size at least 3.

A structure
\begin{cmath}
\A = \struct{S,\lines,\parallel}
\end{cmath}
is a \emph{partial affine partial linear space} iff
$\struct{S,\lines}$ is a partial linear space and
$\parallel$ is an equivalence relation on $\lines$ such that
$L \adjac K$ and $L \parallel K$ implies that $L = K$
for all $L, K\in\lines$.
A partial affine partial linear space $\A$ is an \emph{affine partial linear space} (cf.\ \cite{apls})
when
for all $a\in S, L\in\lines$ there is $K\in\lines$ such that 
$a \in K\parallel L$.
A partial affine partial linear space $\A$ is said to
satisfy \emph{the Tamaschke Bedingung} when
\begin{multline}
  \text{\it for any two lines }\, L_1, L_2\, \text{\it through a point } p\
  \text{\it and any two other lines }\\
  K_1, K_2\, \text{\it not through } p\
  \text{\it if }\, K_1\adjac L_1, L_2,\ K_2\adjac L_1,\ K_1\parallel K_2\
  \text{\it then }\, L_2 \adjac K_2,
\end{multline}
and it is said to satisfy the \emph{parallelogram completion condition} when
\begin{multline}
  \text{\it for any two pairs of parallel lines }
    L_1\parallel L_2,\ K_1 \parallel K_2
  \\
  \text{\it if }\, L_1, L_2 \adjac K_1\ \text{\it and}\
  L_1 \adjac K_2, \text{ then } L_2 \adjac K_2.
\end{multline}
Observe that an \emph{affine space} is an affine linear space which satisfies
the Tamaschke Bedingung and the parallelogram completion condition.


\subsection{Segre products}

Let $I$ be a countable set ($2\leq\abs{I}$) and
let $\M_i = \struct{S_i, \lines_i}$ be a partial linear space for $i\in I$.
Take
\begin{cmath}
  S := \bigtimes_{i\in I} S_i.
\end{cmath}
To make notation easier we apply the following convention: given
$a=(a_1,a_2\ldots)\in S$ and $i\in I$ for a point $x\in S_i$ we write
\begin{cmath}
  \subst(a,i,x) := (a_1, \dots,a_{i-1}, x, a_{i+1}, \dots);
\end{cmath}
for a set $A\subseteq S_i$ we write
\begin{cmath}
  \subst(a,i,A) :=
    \{ (a_1, \dots, a_{i-1}) \}\times A \times \{ (a_{i+1}, \dots) \};
\end{cmath}
for a family ${\cal F} = \{A_j\subseteq S_i\colon j\in J\}$ of subsets of $S_i$,
$J$ being some set of indices,
we write
\begin{cmath}
  \subst(a,i,{\cal F}) :=
    \Bigl\{ \{ (a_1, \dots, a_{i-1}) \}\times A_j \times \{ (a_{i+1}, \dots) \}\colon j\in J \Bigr\}.
\end{cmath}
Now take
\begin{cmath}
  \lines := \bigcup_{i\in I} \Bigl\{ \subst(a,i,\lines_i)\colon a\in S \Bigr\}.
\end{cmath}
The structure
$$
  \bigotimes_{i\in I}\M_i= \struct{S, \lines}
$$
will be called \emph{the Segre product} of $\M_i$. We say that a line $L$ in this
product arises as a line $l$ of $\M_i$ if $L = \subst(a, i, l)$ for some $a\in S$, $i\in I$.
Based on \cite{naumo} let us recall some simple facts.

\begin{fact}\label{fct:prodpls:gener}
  Let\/ $\M = \bigotimes_{i\in I}\M_i$ be the Segre product of partial linear spaces $\M_i$.
  \begin{sentences}\itemsep-2pt
  \item\label{prodpls:axiom}
    The product $\M$ is a partial linear space.
    The connected component of a point $a$ of\/ $\M$ is the set
    $\bset{ x\in S\colon \babs{i\colon x_i\neq a_i} < \infty }$.
    Consequently, $\M$ is connected whenever $I$ is finite.
  \item\label{prodpls:triangle}
    A triangle in\/ $\M$ has the form $\subst(a,i,T_i)$
    for some $a\in S$, $i\in I$, and a triangle $T_i$ in $\M_i$. 
  \item\label{prodpls:strongsub}
    A strong subspace of\/ $\M$ has the form $\subst(a,i,X_i)$
    for some $a\in S$, $i\in I$, and a strong subspace $X_i$ in $\M_i$.
  \item\label{prodpls:gamma}
    If all the\/ $\M_i$ are gamma spaces, then\/ $\M$ is a gamma space.
  \item\label{prodpls:veblen}
    If all the\/ $\M_i$ are Veblenian, then\/ $\M$ is Veblenian.
  \end{sentences}
\end{fact}

If $\A_i = \struct{S_i,\lines_i,\parallel_i}$
is a partial affine partial linear space for $i\in I$
we define the Segre product
$\A=\bigotimes_{i\in I}\A_i = \struct{S,\lines,\parallel}$
so that
$\struct{S,\lines} = \bigotimes_{i\in I}\struct{S_i,\lines_i}$
and for lines $L, K$ of $\A$ we have
\begin{equation}\label{def:paral1:prod}
  L \parallel K \quad:\iff\quad (\exists i\in I)[\, L_i \parallel_i K_i\,]
\end{equation}
$L_i, K_i$ being $i$-th coordinate of $L, K$ respectively.
For further applications let us define a parallelism $\prodparal$ on $\lines$
by the following formula,
much more, in fact, in the spirit of a product
\begin{equation}\label{def:paral2:prod}
  L\prodparal K \quad:\iff\quad (\forall i\in I)[\, L_i = K_i \Lor L_i \parallel_i K_i \,].
\end{equation}

\begin{fact}[{\upshape cf.\ \cite{apls}}]\label{fct:prodapls:gener}
  \begin{sentences}\itemsep-2pt
  \item
    The Segre product\/ $\A$ is a partial affine  partial linear space.
    Its connected components are as in {\upshape \ref{fct:prodpls:gener}\eqref{prodpls:axiom}}.
  \item
    If the\/ $\A_i$ are affine partial linear spaces, then\/ $\A$ is also
    an affine partial linear space.
  \item
    If\/ $\A_i$ satisfies the Tamaschke Bedingung for $i\in I$,
    then $\A$ also satisfies this condition.
  \item
    If\/ $\A_i$ satisfies the parallelogram completion condition for $i\in I$,
    then\/ $\A$ also satisfies this condition.
  \end{sentences}
\end{fact}

\begin{rem}\label{rem:notapls}
  The structure $\A=\struct{S,\lines,\prodparal}$ is a partial affine partial linear space 
  but it is not an affine partial linear space.
\end{rem}

\begin{proof}
  Clearly $\A$ is a partial linear space and $\prodparal$ is an equivalence relation on $\lines$.
  Let $i\in I$, $a\in S$, $l\in\lines_i$ and $L=\subst(a,i,l)$.
  Suppose that there is $K\in\lines$ such that $L \adjac K$ and $L\prodparal K$. Then $K=\subst(a,i,k)$ for
  some line $k\in\lines_i$, $k\parallel_i l$ and lines $k$, $l$ share a point. It follows that $k=l$
  and consequently $K=L$, so $\A$ is partial affine.

  There is however $b\in S$ such that for some $i_1,i_2\in I$
  we have $b_{i_1},b_{i_2}\neq a_i$ for all
  $i\in I$. Hence no line through $b$ is parallel to $L$, and thus $\A$ is not affine.
\end{proof}


\section{Affinization of partial linear spaces}

Let $\M = \struct{S,\lines}$ be a partial linear space and $\hipa$ its hyperplane.
We write $\lines^\propto = \left\{ L\in\lines\colon L\nsubseteq\hipa \right\}$.
For each $L\in\lines^\propto$ there is a unique point $L^\infty\in\hipa$ with
$L^\infty\in L$. This enables us to define a natural parallelism $\parallel_{\hipa}$
on $\lines^\propto\times\lines^\propto$ by the following condition
\begin{equation}\label{eq:hparallelism}
  L \parallel_{\hipa} K \iff L^\infty = K^\infty.
\end{equation}
The complement of $\hipa$ in $\M$ is
\begin{equation}\label{def:afinizacja}
  \M\setminus\hipa = \struct{S \setminus \hipa, \lines^\propto, \parallel_{\hipa}}.
\end{equation}
With straightforward reasoning we get the following.

\begin{fact}\label{fct:afred:gener}
  Let\/ $\M$ be a partial linear space and let\/ $\hipa$ be its hyperplane.
  \begin{sentences}\itemsep-2pt
  \item\label{afredgen:papls}
    The complement $\A = \M\setminus\hipa$ is a partial affine partial linear space.
  \item\label{afredgen:apls}
    If\/ $\hipa$ is \kolczaty, then the complement\/ $\A$ is an affine partial linear space iff
    $a \adjac b$ for all $a\in\hipa$ and $b\notin \hipa$.
  \item\label{afredgen:ls}
    If\/ $\M$ is a linear space, then\/ $\A$ is an affine linear space.
  \item\label{afredgen:veb2tam}
    If\/ $\M$ is Veblenian, then\/ $\A$ satisfies the parallelogram completion
    condition and the Tamaschke Bedingung.
  \end{sentences}
\end{fact}

\begin{rem}
  The converse of \ref{fct:afred:gener}\eqref{afredgen:ls} is false, in general.
\end{rem}

\begin{proof}
  Take any linear space $\M = \struct{S,\lines}$
  with a hyperplane $\hipa$ (e.g. let $\M$ be
  a classical projective space). Consider the set
  $\lines(\hipa) = \{ L\in\lines\colon L\subseteq\hipa \}$,
  then $\M' = \struct{S,\lines\setminus\lines(\hipa)}$ is not a linear space.
  Nevertheless, $\hipa$ is a hyperplane of $\M'$ and
  $\M'\setminus\hipa = \M\setminus\hipa$ is a linear space.
\end{proof}

Affinization may break vital properties like connectedness.
 
\begin{exm}\label{exm:affinconnected}
  Let $Y_1,Y_2$ be two projective $k$-subspaces of a projective space $\goth P$
  such that $\hipa=Y_1\cap Y_2$ has dimension $k-1\geq 1$. Let $\M$ be the
  restriction of $\goth P$ to $Y_1 \cup Y_2$. Then $\M$ is a 
  strongly connected Veblenian gamma space and $\hipa$ is a \luskwiaty\ hyperplane
  in $\M$. However, $\M\setminus\hipa$ is not connected, and thus not strongly
  connected.
\end{exm}

It may also break the property of being \kolczaty\ or being \luskwiaty. That is,
if $\hipa$ is a \luskwiaty\ hyperplane in $\M$, then its restriction to 
a substructure $\struct{S_0, \lines_0}$, in the sense of \ref{lem:hip:restricted},
may be non-\kolczaty, and consequently non-\luskwiaty, hyperplane in that substructure. 

There is a natural correspondence between strong subspaces of hyperplane
complements in Veblenian gamma spaces and strong subspaces of the respective
ambient spaces.

\begin{lem}\label{prop:strong:afprodPLS}
  Let\/ $\M$ be a Veblenian gamma space with lines of size at least 4 and let $\hipa$
  be a hyperplane in\/ $\M$.
  A set\/ $X\subseteq S$ is a strong subspace in\/ $\M\setminus\hipa$ iff there is
  a strong subspace\/ $Y$ in\/ $\M$ such that\/ $X=Y\setminus\hipa$.
\end{lem}

\begin{proof}
  \ltor
  Set $\lines^\propto(X) := \{ L\in\lines^\propto\colon \abs{L\cap X}\ge 2 \}$.
  Let us begin proving that 
  \begin{ctext}
  $(*)$ given $a\in X$ and $L\in\lines^\propto(X)$
  there is a line $K\in\lines^\propto(X)$ such that $a\in K\parallel_\hipa L$.
  \end{ctext}
  We drop the trivial case where $a\in L$ and assume that $a\notin L$. 
  Take two distinct points $p, q\in L\cap X$.
  As $X$ is strong we have $a\adjac p, q$. Hence by none-one-or-all axiom $a\adjac L^\infty$
  in $\M$. Set $K := \LineOn(a, L^\infty)$. Now take a point $r\in\LineOn(a, p)\cap X$
  distinct from $a, p$. The line $\LineOn(q, r)$
  intersects two sides of the triangle $a, p, L^\infty$ in $\M$ so, it intersects
  $K$ in some point $b$ by the Veblen condition. Note that $q, r\in X$ and
  $L^\infty\notin\LineOn(q,r)$ thus $b\in X$. Finally, $K = \LineOn(a, b)\in\lines^\propto(X)$
  and thus $a\in K\parallel_\hipa L$.
 
  Now, set $X^\infty := \bset{L^\infty\colon L\in\lines^\propto(X)}$ and $Y := X \cup X^\infty$.
  We will show that $Y$ is a strong subspace in $\M$.
  
  Let $u\in X^\infty$ and $a\in X$. Note that $u = L^\infty$ for some $L\in\lines^\propto(X)$.
  Then by $(*)$ we have $a\adjac u$ and $\LineOn(a, u)\subseteq Y$.
  
  Now, let $u, w\in X^\infty$. There are $L, M\in\lines^\propto(X)$ with
  $u = L^\infty$ and $w = M^\infty$.
  Take $a, b\in M\cap X$. By $(*)$ we get $u\adjac a, b$. Then, by none-one-or-all axiom
  $u\adjac w$. What is left is to show that $\LineOn(u,w)\subseteq Y$. Take a point
  $v\in\LineOn(u,w)$ distinct from $u, w$. Hence $v\in\hipa$.
  As $a\adjac u, w$ we get $a\adjac v$  by none-one-or-all axiom. 
  Take a point $p\in\LineOn(u, a)$ distinct from $u, a$. The line $K := \LineOn(v, a)$
  intersects two sides: $\LineOn(u, w)$ and $\LineOn(u, p)$, of the triangle $u, w, p$. 
  Hence, by the Veblen condition, it intersects $\LineOn(w, p)$ but not in $w$ or $p$
  as otherwise we would have $v=w$ or $v=u$, respectively, which is impossible. Therefore 
  $\abs{K\cap X}\ge 2$ and thus $v\in K\subseteq Y$.
  
  \rtol
  Immediate by the definition of a strong subspace.
\end{proof}


\subsection{Recovering}

Right from the definitions \eqref{eq:hparallelism}, \eqref{def:afinizacja},
assuming that $\hipa$ is \kolczaty,  the deleted
points on $\hipa$ can be identified with the equivalence classes of $\parallel_{\hipa}$
i.e.\ with the elements of $\lines^\propto/\mathord{\parallel_{\hipa}}$.
To recover $\M$ from its affine reduct
$\A = \M\setminus\hipa$ we need also to determine in terms of $\A$
the (ternary) collinearity relation on $\lines^\propto/\mathord{\parallel_{\hipa}}$.
In short, this is usually achieved by use of planes in $\A$ that intersect $\hipa$. 
This is a very rough approximation of what should be done.
The root of the problem is to determine assumptions under which this recovering procedure 
can be implemented.

The points $a_1, a_2,\dots$ of $\M$ are said to be \emph{collinear} when they all
are on a line of $\M$, in symbols $\colin(a_1,a_2,\dots)$.

\begin{prop}\label{prop:af2horlines}
  Let\/ $\M$ be a Veblenian gamma space with lines of size at least $3$.
  If\/ $\hipa$ is a \luskwiaty\ hyperplane in\/ $\M$, then
  for all pairwise distinct points $p_1,p_2,p_3\in\hipa$ we have
  \begin{equation}
    \colin(p_1,p_2,p_3) \iff (\exists a_1,a_2,a_3\in S\setminus\hipa)
    \bigl[\, \land_{\neq(i,j,k)} \bigl(a_i \adjac a_j \Land p_k\in\LineOn(a_i,a_j)\bigr) \,\bigr].
  \end{equation}
  The lines of\/ $\hipa$ are defined in an abstract way
  as the equivalence classes of the relation
  \begin{multline}
    \colin\bigl([L_1]_\parallel, [L_2]_\parallel, [L_3]_\parallel\bigr) \iff
      \text{ there is a triangle with the sides } K_1, K_2, K_3
      \\
      \text{ such that } K_i \parallel L_i \text{ for all\/ } i=1,2,3 \text{ in }\M\setminus\hipa. 
  \end{multline}
\end{prop}

\begin{proof}
  \ltor
  Let $L$ be the line through $p_1,p_2,p_3\in\hipa$. Since $\hipa$ is \luskwiaty\
  there is a point $a \notin\hipa$ such that $a\adjac p_i$
  for every $i=1,2,3$. Take a point $b\in \LineOn(a,p_1)$
  distinct from $a,p_1$. Since $a \notin\hipa$ we get $b\notin\hipa$ as
  well. We have $p_2\adjac a,p_1$ and thus $p_2\adjac b$. Let
  $K=\LineOn(b,p_2)$. From the Veblen condition we get a point $c\in
  K\cap\LineOn(a,p_3)$ of $\M\setminus\hipa$. Finally $a,b,c\in S\setminus\hipa$ form the required
  triangle.

  \rtol
  As $\M$ is a gamma space, from $a_2\adjac a_3,a_1$ we get
  $a_2\adjac p_2$, and then $p_2\adjac a_3,a_2$ yields
  $p_2\adjac p_1$. The line $L=\LineOn(p_1,p_2)$ meets
  $\LineOn(a_1,a_2)$ in a point $q$, that follows from the Veblen condition.
  Furthermore $L\subseteq\hipa$, thus $q\in\hipa$. If $q\neq p_3$
  then $\LineOn(q,p_3)\subseteq\hipa$ and $a_1,a_2\in\hipa$ in
  particular, that contradicts the assumptions. Therefore $q=p_3$ and
  points $p_1,p_2,p_3$ are collinear.
\end{proof}

It is more than likely that the above method that relies on flappy property of hyperplanes
to recover the ambient space is not unique and there are other procedures that could be applied.
We do not want however to go any deeper into discussion of possible methods.


\subsection{Automorphisms}\label{sec:auty:pls}

Let $\hipa$ be a hyperplane of a partial linear space $\M = \struct{S,\lines}$.
For $p\in S$, $K\in\lines$ we define
\begin{cmath}
  \Pi(p,K):=\bigcup\bset{ M\in\lines\colon p\in M, \; M\cap K\neq\emptyset }.
\end{cmath}
Two possibilities arise:
either $p\in K$ and then $\Pi(p,K)$ is the set of all points on lines through $p$,
or $p\notin K$ and then $\Pi(p,K)=\emptyset$ or $\Pi(p,K)$ is the set of all points on
lines through $p$ that cross $K$ not in $p$.
We call $\Pi(p,K)$ a \emph{near-plane} of $\M$ if $p\notin K$,
$\Pi(p,K)\neq\emptyset$, and $\Pi(p,K)$ is  not a single line.

The following is just a standard exercise.
\begin{fact}\label{fct:autyaf:gener}
  Let $\hipa$ be a hyperplane of a partial linear space $\M = \struct{S,\lines}$
  and let $\A = \M\setminus\hipa$.
  \begin{sentences}\itemsep-2pt
  \item
    If\/ $F$ is an automorphism of\/ $\M$ that preserves $\hipa$, then
    $F\restriction(S\setminus\hipa)$ is an automorphism of\/ $\A$.
  \item
    Let $f\in\Aut({\goth A})$.
    \newdimen\oldleftmarginii
    \setlength{\oldleftmarginii}{\leftmarginii}
    \setlength{\leftmarginii}{17mm}
    \begin{enumerate}[{\upshape (a)}]
    \item\label{fct:autyaf:gener:a}
      If\/ $\hipa$ is \kolczaty, then $f$ extends to a bijection $F$ of\/ $S$ determined
      by the conditions:
      \begin{ctext}
        $F(x) = f(x)$ for $x\in S\setminus\hipa$,\qquad $F(L^\infty) = f(L)^\infty$ for every line $L$ of\/ $\A$.
      \end{ctext}
    \item\label{fct:autyaf:gener:b}
      If\/ $\hipa$ satisfies the condition:
      \begin{multline}\label{zal:przecinanie}
        \text{for every point } p\notin\hipa \text{ and every line } K\nsubseteq\hipa\\
          \text{the near-plane }\Pi(p,K) \text{ of }\; \M \text{ meets } \hipa \text{ in a line},
      \end{multline}
      and the following analogue of \luskwiaty\ condition:
       \begin{multline}\label{zal:nearlyluskw}
         \text{for every line } L\subseteq\hipa \text{ there is a near-plane }
           \Pi(p,K) \text{ of }\; \M\\
             \text{containing } L \text{ such that } p\notin\hipa
               \text{ and } K\nsubseteq\hipa
      \end{multline}
     then $F$ from {\upshape \eqref{fct:autyaf:gener:a}} is an automorphism of\/ $\M$ preserving $\hipa$.
    \end{enumerate}
  \end{sentences}
\end{fact}

\begin{lem}\label{lem:veblgamma}
  Let\/ $\M$ be a Veblenian gamma space with lines of size at least\/ $3$.
  \begin{sentences}\itemsep-2pt
  \item
     Hyperplanes of\/ $\M$ satisfy \eqref{zal:przecinanie}.
  \item
     Flappy hyperplanes of\/ $\M$ satisfy \eqref{zal:nearlyluskw}.
  \end{sentences}
\end{lem}
\begin{proof}
  Let $\hipa$ be a hyperplane in $\M$.

  (i)
  Let $\Pi(p,K)$ be a near-plane of $\M$ with $p\notin\hipa$ and $K\nsubseteq\hipa$.
  Since $\M$ is a gamma space and the size of $K$ is at least $3$,
  by definition the near-plane $\Pi(p,K)$ contains $3$ distinct lines $L_1,L_2,L_3$
  such that $p\in L_1,L_2,L_3$.
  Let $a_i\in L_i\cap K$, $b_i\in L_i\cap\hipa$ for $i=1,2,3$.
  As $\M$ is a gamma space we get $a_1\adjac b_2$, and next $b_1\adjac b_2$.
  Let $\LineOn(b_1,b_2)=L$. Assume also $b_3\notin L$. None-one-or-all axiom gives
  $b_2\adjac b_3$ and $b_3\adjac b_1$. Denote $\LineOn(b_2,b_3)=L'$ and $\LineOn(b_3,b_1)=L''$.
  Then, from the Veblen condition, $K$ meets lines $L,L',L''$ in at least two distinct points.
  These points are on $\hipa$, that contradicts  $K\nsubseteq\hipa$.
  Thus $b_3\in L$, and consequently $\hipa$ satisfies \eqref{zal:przecinanie}.

  (ii)
  Assume that $\hipa$ is \luskwiaty\ and
  let $L$ be a line contained in $\hipa$. There is a point $p\notin\hipa$
  such that $L\subseteq [p]_{\mathord{\adjac}}$. Since $\abs{L}>2$ there are two distinct lines
  $L_1,L_2$ such that $p\in L_1,L_2$ and $L_i\cap L\neq\emptyset$ for $i=1,2$.
  Let $a_i\in L_i\cap L$ and take $b_i\in L_i$ with $b_i\neq a_i,p$ for $i=1,2$.
  From none-one-or-all axiom $a_1\adjac b_2$, and then $b_2\adjac b_1$ as well.
  Denote $\LineOn(b_1,b_2)$ by $K$. Note, that $K\nsubseteq\hipa$ as $b_1,b_2\notin\hipa$.
  As ${\goth M}$ is Veblenian, we have $K\cap L\neq\emptyset$. Let us consider the
  near-plane $\Pi(p,K)$ and a point $a_3\in L$. If $a_3=a_1$, $a_3 = a_2$ or $a_3\in K\cap L$,
  then immediately $a_3\in \Pi(p,K)$. Otherwise we take a line $L_3:=\LineOn(a_3,p)$.
  The lines $K,L_3$ intersect $L,L_1$ so that there are $4$ distinct points of intersection. 
  Hence, the Veblen
  condition yields $K\cap L_3\neq\emptyset$. So, we get $a_3\in \Pi(p,K)$,
  and thus $L\subseteq \Pi(p,K)$.
\end{proof}
From \ref{fct:autyaf:gener} and \ref{lem:veblgamma} we obtain

\begin{cor}\label{cor:autext}
  If\/ $\M$ is a Veblenian gamma space with lines of size at least\/ $3$
  and\/ $\hipa$ is a \luskwiaty\/ hyperplane of\/ $\M$, then every automorphism
  of\/ $\M\setminus\hipa$ can be uniquely extended to an automorphism of\/ $\M$.
\end{cor}

In view of \ref{exm:affinconnected}, connectedness of $\M$ need not to imply
connectedness of $\M\setminus\hipa$. Therefore, an essential tool to redefine
$\hipa$ in terms of its complement in $\M$, that is to extend an automorphism
of $\M\setminus\hipa$ to $\M$, is the parallelism $\parallel_{\hipa}$. That is
why one needs to be aware that \ref{cor:autext}, as well as forthcoming 
\ref{prop:redefprod} and \ref{thm:auty:prod}, are false for a hyperplane complement
considered as a point-line incidence structure without parallelism.


\section{Affinization of Segre products}


\subsection{Hyperplanes in Segre products}

Let $\M_i = \struct{S_i,\lines_i}$ be a partial linear space,
and let $\hipy_i$ be the family of all hyperplanes in\/ $\M_i$ and the point set $S_i$ for $i\in I$.
Set $S := \bigtimes_{i\in I} S_i$ and $\M := \bigotimes_{i\in I}\M_i$.
Consider a hyperplane $\hipa$ in $\M$. We will write
\begin{equation}\label{eq:hipaslice}
  \hipa^{[a]}_i := \bset{ x\in S_i\colon \subst(a, i, x)\in\hipa }.
\end{equation}
for a point $a\in S$ and $i\in I$.

\begin{thm}\label{thm:hip:inprod}
  For $\hipa\subseteq S$ the following conditions are equivalent:
  \begin{conditions}\itemsep-2pt
  \item\label{hipinprod:war1}
    $\hipa$ is a hyperplane in\/ $\M$.
  \item\label{hipinprod:war2}
    For all $a\in S$ and $i\in I$ we have $\hipa^{[a]}_i\in{\cal H}_i$ but
    $\hipa^{[a]}_i \neq S_i$ for some $a\in S$ and $i\in I$.
  \end{conditions}
\end{thm}

\begin{proof}
  To justify the equivalence of \eqref{hipinprod:war1} and \eqref{hipinprod:war2}
  it suffices to consider
  the sets $\subst(a,i,S_i)$ for arbitrary $a\in S$ and $i\in I$, which are subspaces of $\M$.
  Note that either $\hipa\cap\subst(a,i,S_i)$ is the whole of $\subst(a,i,S_i)$
  or a hyperplane in it.
  Clearly, for fixed $a\in S$ and $i\in I$
  the map $S_i\ni x_i\mapsto\subst(a,i,x_i)$
  is an isomorphism of $\M_i$ onto $\subst(a,i,S_i)$.
  Therefore, there is $X\in\hipy_i$ such that
  $\subst(a,i,S_i)\cap\hipa = \subst(a,i,X)$.
  It is seen that $X = \hipa^{[a]}_i$.
\end{proof}

In particular case of a product of two spaces \ref{thm:hip:inprod} can be worded
in terms of a correlation.

\begin{rem}\label{rem:correl}
Let $I = \set{1,2}$. The set $\hipa\subseteq S$ is a hyperplane in\/ $\M$ iff
there are two maps: $\delta_i\colon S_i\lto \hipy_{3-i}$ such that
\begin{cmath}
  \delta_i(a_i) = \hipa_{3-i}^{[a]}
  \quad\text{and}\quad
  \delta_{3-i}(a_{3-i}) = \bset{a_i\in S_i\colon a_{3-i}\in\delta_i(a_i)}
\end{cmath}
for all $a=(a_1,a_2)\in S_i\times S_2$, $i\in I$ 
and there is $a_i\in S_i$ with $\delta_i(a_i)\neq S_{3-i}$ for some $i\in I$.
Moreover, if $\hipa$ is a hyperplane then
\begin{cmath}
  \hipa = \bset{(a_1, a_2)\colon a_1\in S_1, a_2\in\delta_1(a_1)} =
            \bset{(a_1, a_2)\colon a_2\in S_2, a_1\in\delta_2(a_2)}.
\end{cmath}
\end{rem}

\begin{rem}\label{rem:reduct:notAPLS}
  If\/ $\hipa$ is a \kolczaty\ hyperplane in the product $\M$ of partial linear spaces 
  on at least three points each, then $\M\setminus\hipa$ \emph{is not} an affine partial 
  linear space (cf. \ref{fct:afred:gener}\eqref{afredgen:papls}).
\end{rem}

\begin{proof}
  Note first that for a point $a$ in $\M$ we have 
  $[a]_{\mathord{\adjac}}\subseteq\bigcup_{i\in I}\subst(a,i,S_i)$.
  Suppose to the contrary that $\M\setminus\hipa$ is an affine partial linear space.
  Let $p\in\hipa$. In view of \ref{fct:afred:gener}\eqref{afredgen:apls}, all 
  the points non-collinear with $p$ lie on $\hipa$.
  In that case $x_i \neq p_i$ for all $i$ yields $x \in \hipa$, for every $x \in S$.
  From the assumptions, there exists $p'\in S$ such that $p'_j\neq p_j$
  for all $j\in I$. Let $i \in I$. There exists $x\in S_i$ with $x \neq p_i,p'_i$.
  Take $q \in S$ with $q_i = x$, $q_j = p_j$ for $j\neq i$.
  We have, consecutively, $p'\in\hipa$ (as $p$ and $p'$ differ on all of the coordinates),
  and $q \in \hipa$ (as $p'$ and $q$ differ on all of the coordinates).
  So, we get that: if $p\in\hipa$ and
  $\babs{\set{ i\in I\colon p_i = p'_i }}\leq m$ with $m=1$,
  then $p'\in\hipa$. Inductively, we can enlarge $m$ and
  finally we get $\hipa = S$, a contradiction.
\end{proof}

Although complements of \kolczaty\ hyperplanes are not affine partial linear spaces these 
hyperplanes remain beneficial for affinization: all the points of a \kolczaty\ hyperplane $\hipa$
are directions of the parallelism $\parallel_{\hipa}$ in $\M\setminus\hipa$.


\subsubsection{Non-degenerate hyperplanes}

In recovering the Segre product from the complement of its hyperplane we heavily
rely on the \luskwiaty\ property of that hyperplane. It will be shown later that
this property is related to another intrinsic property of hyperplanes.

A hyperplane $\hipa$ of a Segre product of partial linear spaces is called
\emph{non-degenerate} when $\hipa^{[a]}_i$ is a hyperplane for all $a\in S$ and $i\in I$.
In the context of \ref{rem:correl} we can say that $\hipa$ is non-degenerate if
both $\delta_1$ and $\delta_2$ take hyperplanes as their values.

Main properties of non-degenerate hyperplanes of the Segre product come into 
hyperplanes of its components.
 
\begin{fact}\label{lem:prezerw:wystaw}
  Let\/ $\hipa$ be a hyperplane of\/ $\M$.
\begin{sentences}\itemsep-2pt
\item\label{lem:prezerw:wystaw:i}
  Assume that $\hipa$ is non-degenerate. Then $\hipa$ is \luskwiaty\ iff\/ $\hipa^{[a]}_i$ is a \luskwiaty\
  hyperplane in\/ $\M_i$ for all $a\in S$ and\/ $i\in I$.
\item\label{lem:prezerw:wystaw:ii}
  The hyperplane $\hipa$ is \kolczaty\ iff\/ for every $a\in S$ there is \/ $i\in I$ 
  such that $\hipa^{[a]}_i$ is a \kolczaty\ hyperplane in\/ $\M_i$.
\end{sentences} 
\end{fact}

\begin{proof}
  Only the right-hand part of the equivalence in \ref{lem:prezerw:wystaw}\eqref{lem:prezerw:wystaw:ii} 
  seems to be not evident. Suppose that there is a point $a\in S$ such
  that either $\hipa^{[a]}_i=S_i$ or $\hipa^{[a]}_i$ is non-\kolczaty\ for all $i\in I$. 
  Thus, for every $i\in I$ 
  there is a point $b_i\in\hipa^{[a]}_i$, which is collinear with no point in
  $\M_i\setminus\hipa^{[a]}_i$. Then $b=(b_1,b_2,\dots)\in\hipa$ is collinear
  with no point in $\M\setminus\hipa$, so $\hipa$ is non-\kolczaty.
\end{proof}

Immediately from 
\ref{lem:prezerw:wystaw} we obtain

\begin{cor}\label{cor:wystawLS}
  Let\/ $\M_i$ be a linear space for all $i\in I$ and let\/ $\hipa$ be a hyperplane in\/ $\M$.
  \begin{sentences}\itemsep-2pt
  \item\label{cor:wystawLS:i}
    The hyperplane $\hipa$ is \luskwiaty\ iff\/ $\hipa$ is non-degenerate.
  \item\label{cor:wystawLS:ii}
    The hyperplane $\hipa$ is \kolczaty\ iff\/ for every $a\in S$ there is \/ $i\in I$ 
    such that $\hipa^{[a]}_i$ is a hyperplane in\/ $\M_i$.
  \end{sentences}
\end{cor}

Now, by \ref{prop:af2horlines}, we can state the following.

\begin{prop}\label{prop:redefprod}
  If\/ $\M$ is a Veblenian gamma space with lines of size at least\/ $3$ and\/
  $\hipa$ is a non-degenerate \luskwiaty\ hyperplane in\/ $\M$, then\/ $\M$ 
  can be defined in terms of\/ $\M\setminus\hipa$.
\end{prop}


\subsubsection{Degenerate hyperplanes}

Degenerate hyperplanes are indeed defective from our view.

\begin{lem}\label{lem:hipynieluskowate}
  Degenerate hyperplanes of\/ $\M$ are not \luskwiaty.
\end{lem}

\begin{proof}
  Let $\hipa$ be a degenerate hyperplane of a Segre product.
  So, there are $a$ and $i$ such that $\hipa^{[a]}_i = S_i$,
  which means that
  $\subst(a,i,S_i) \subseteq \hipa$.
  Let $l\in\lines_i$, then
  $L = \subst(a,i,l)$
  is a line of the product contained in $\hipa$.
  In view of \ref{fct:prodpls:gener}\eqref{prodpls:triangle}
  a triangle in the product with
  $L$
  as one of its sides is contained in $\subst(a,i,S_i)$,
  so no point outside  $\hipa$ can be a vertex of such a triangle.
\end{proof}
There is quite natural construction of a hyperplane in the Segre
product as long as there are hyperplanes in all of the components. 
The outcome, however, is degenerate.
For hyperplanes $\hipa_i$ in $\M_i$, $i\in I$, we write
\begin{equation}\label{eq:degenhipa}
  \otimes_{i\in I}\hipa_i := \bigcup_{i\in I}
    (S_1\times\dots\times S_{i-1}\times\hipa_i\times S_{i+1}\dots).
\end{equation}
To shorten notation let us set $\hipa := \otimes_{i\in I}\hipa_i$.

\begin{prop}\label{prop:degenkolcz}
 The set\/ $\hipa$ is a degenerate and non-\kolczaty\ hyperplane in\/ $\M$.
\end{prop}

\begin{proof}
  Let $a$ be a point of $\M$ and $L_i$ a line in $\M_i$ for some $i\in I$.
  Then $L = \subst(a,i,L_i)$ is a line in $\M$.
  There is a point $a'_i\in L_i\cap\hipa_i$ and thus
  $\subst(a,i,a'_i)$ is a common point of $L$ and $\hipa$.
  Hence $\hipa$ is a hyperplane of $\M$.

  Take a point $b$ with $b_i\in\hipa_i$.
  Then $\hipa^{[b]}_j=S_j$ for any $j\neq i$ and thus $\hipa$ is degenerate.
  Let $d$ be a point of $\M$ with $d_i\in\hipa_i$ for some $i=i_1,i_2\in I$, $i_1\neq i_2$. Clearly $d\in\hipa$.
  Let $d'$ be a point of
  the product collinear with $d$. So, $d'=\subst(d,i,d'_i)$ for some $i\in I$, $d'_i\adjac_i d_i$.
  It is easy to note, that $d'\in\hipa$ for any $i\in I$. Consequently, $\hipa$ is not \kolczaty.
\end{proof}

Observe that the points and the lines of the complement $\M\setminus\hipa$ coincide with 
the points and the lines of the product $\bigotimes_{i\in I} \struct{\M_i\setminus\hipa_i}$.
Since all complements $\M_i\setminus\hipa_i$ for $i\in I$ are partial affine partial linear spaces
by \ref{fct:afred:gener}\eqref{afredgen:papls}, we can apply \eqref{def:paral1:prod}
to define parallelism $\parallel$ on their product
$\bigotimes_{i\in I} \struct{S_i\setminus\hipa_i,\lines^\propto_i}$.
This parallelism however, is not compatible with the parallelism $\parallel_\hipa$ in the
complement $\M\setminus\hipa$. 
As $\hipa$ is a hyperplane introduced in \eqref{eq:degenhipa} the parallelism
$\parallel_{\hipa}$ in $\M\setminus\hipa$  is the relation $\prodparal$ given by
\eqref{def:paral2:prod} with $\parallel_i$=$\parallel_{\hipa_i}$.
This is the subject of the following statement.

\begin{prop}\label{prop:isomorph}
  \begin{sentences}\itemsep-2pt
  \item
    $\bstruct{S \setminus\hipa, \lines^\propto}\cong 
       \bigotimes_{i\in I} \bstruct{S_i\setminus\hipa_i,\lines^\propto_i}$.
  \item
    $\M\setminus\hipa=\bstruct{S \setminus\hipa, \lines^\propto, \parallel_\hipa}\cong 
        \bigl(\bigotimes_{i\in I} \struct{S_i\setminus\hipa_i,\lines^\propto_i}, \prodparal\bigr)$.
  \end{sentences}
\end{prop}

\begin{proof}
  (i)
  It suffices to note that
  \begin{cmath}
    S \setminus\hipa=\bigtimes_{i\in I} S_i \setminus \bigcup_{i\in I}
  (S_1\times\dots\times S_{i-1}\times\hipa_i\times S_{i+1}\dots) =
  \bigtimes_{i\in I} (S_i \setminus\hipa_i).
  \end{cmath}

  (ii)
  Take $L,K\in \lines^\propto$ such that $L\parallel_{\hipa} K$.
  Then $L=\subst(a,i,L_i)$, $K=\subst(b,j,K_j)$ for some $a,b\in S \setminus\hipa$, 
  $L_i\in\lines_i^\propto$, $K_j\in\lines_j^\propto$, and $i,j\in I$.
  It is seen that
  \begin{cmath}   
    \subst(a,i,L_i^\infty)=L^{\infty}=K^{\infty}=\subst(b,j,K_j{^\infty})
  \end{cmath}
  which means that $i=j$, $a_s=b_s$ for all $s\neq i$, and $L_i\parallel_{\hipa_i} K_i$.
  This reasoning can be easily reversed.
\end{proof}


\subsection{Strong subspaces}

Directly from \ref{fct:prodpls:gener}\eqref{prodpls:gamma}, \ref{fct:prodpls:gener}\eqref{prodpls:veblen} 
and \ref{prop:strong:afprodPLS} we have

\begin{lem}\label{cor:affofprod:strong}
  Let\/ $\M_i$ be a Veblenian gamma space with lines of size at least 4 for $i\in I$ 
  and let\/ $\hipa$ be a hyperplane in\/ $\M$. 
  A set\/ $X\subseteq S$ is a strong subspace in\/ $\M\setminus\hipa$
  iff there is a strong subspace\/ $Y$ in\/ $\M$ such that\/ $X=Y\setminus\hipa$.
\end{lem}

This lets us get a more detailed characterization of strong subspaces 
in the complement of a product.
\begin{prop}\label{prop:affofprod:strong}
  Let\/ $\M_i$ be a Veblenian gamma space with lines of size at least 4 for $i\in I$ 
  and let\/ $\hipa$ be a hyperplane in\/ $\M$. 
  For $X\subseteq S$ the following conditions are equivalent.
  \begin{conditions}\itemsep-2pt
  \item
    $X$ is a strong subspace of the complement\/ $\M\setminus\hipa$.
  \item
    $X=\subst(a,i,Y_i)\setminus\hipa$ for some $a\in S$, $i\in I$, 
    and a strong subspace $Y_i$ in\/ $\M_i$.
  \item
    $X=\subst(a,i,X_i)$ for some $a\in S$, $i\in I$, 
    and a strong subspace $X_i$ in\/ $\M_i\setminus\hipa_i^{[a]}$.
  \end{conditions}
\end{prop}

\begin{proof}
  (i)$\implies$(ii)
  According to \ref{cor:affofprod:strong} we have a strong subspace $Y$ in $\M$ 
  such that $X = Y\setminus\hipa$.
  By \ref{fct:prodpls:gener}\eqref{prodpls:strongsub} $Y=\subst(a,i,Y_i)$ for some $a\in S$, 
  $i\in I$, and a strong subspace $Y_i$ in $\M_i$. Thus 
  $X=\subst(a,i,Y_i)\setminus\hipa$.
  
  (ii)$\implies$(iii)
  For $X=\emptyset$ it suffices to take $X_i=\emptyset$, so assume that $X\neq\emptyset$.
  Hence $\hipa_i^{[a]}$ is a hyperplane in $\M_i$.
  Taking $X_i := Y_i\setminus\hipa_i^{[a]}$ we are through by \ref{prop:strong:afprodPLS}.
  
  (iii)$\implies$(i)
  Immediate by the definition of a strong subspace.
  \end{proof}
By \ref{fct:prodpls:gener}\eqref{prodpls:veblen} and \ref{fct:afred:gener} we infer

\begin{cor}\label{cor:strongasaffine}
  Let\/ $\M_i$ be a Veblenian gamma space with lines of size at least 4 for $i\in I$.
  If\/ $\hipa$ is a hyperplane in $\M$,
  then the complement\/ $\M\setminus\hipa$ satisfies the parallelogram
  completion condition and the Tamaschke Bedingung. Consequently, its
  strong subspaces are affine spaces.
\end{cor}


\subsection{Parallelism in terms of incidence}\label{sec:parinterms}

Most of the time, also in a hyperplane complement of a Segre product,
parallelism can be defined in terms of incidence using the Veblen configuration
as follows. Let $L_1, L_2\in\lines$.
\begin{multline}\label{eq:veblparal}
  L_1 \veblparal L_2 \iff
     L_1 = L_2 \Lor
    \text{ there are lines } K_1, K_2 \text{ through a point } p
    \\
    \text {such that }
    p \not\in L_1, L_2 \Land L_1, L_2 \adjac K_1, K_2 \Land L_1 \not\adjac L_2.
\end{multline}
There is however parallelism typical to a Segre product.
\begin{multline}\label{eq:astparal}
  L_1 \astparal L_2 \iff
     \text{ there exists a quadrangle } p,q,r,s \text{ without diagonals }
     \\
     \text{ such that}\ L_1 = \LineOn(p,q) \Land L_2 = \LineOn(r,s).
\end{multline}
It is based on the fact that four lines $L_1, L_2, K_1, K_2$ form a quadrangle
without diagonals in $\M$ if they all arise as lines $l_1, l_2, k_1, k_2$ 
of some component $\M_i$ where $l_1, l_2, k_1, k_2$ form a quadrangle without 
diagonals, or they are from different components. If the latter is the case,
inspecting coordinates carefully one can easily see that the opposite sides
are always disjoint.
Note that if $\M_i$ are line spaces, then only the latter holds true.

In an affine space the parallelism can be defined in terms of the incidence.
The same can be done in the Segre product of affine spaces.
\begin{thm}\label{thm:defofparal}
  Let\/ $\A_i = \struct{S_i,\lines_i,\parallel_i}$ be an affine space and\/
  $\A'_i=\struct{S_i,\lines_i}$ for $i=1,2$.
  The parallelism $\parallel$ of the Segre product
  $\A_1\otimes\A_2$ can be defined in terms of the point-line incidence , i.e.\ in
  $\A'_1\otimes\A'_2$.
\end{thm}

\begin{proof}
  Let $L_1,L_2$ be lines of the product $\A := \A_1\otimes\A_2$.
  It suffices to observe that in view of 
  \eqref{eq:veblparal} and \eqref{eq:astparal} both $\veblparal$ and $\astparal$ are
  definable in $\A'_1\otimes\A'_2$, and that the following three facts hold true.
  \begin{enumerate}[a)]
  \item
    $L_1 \prodparal L_2$ iff $L_1 \veblparal L_2$.
  \item
    If $L_1 \astparal L_2$, then $L_1 \parallel L_2$.
  \item
    $L_1,L_2$ are parallel in $\A$ iff
    there is a line $L_3$ with $L_1 \prodparal L_3 \astparal L_2$.
  \end{enumerate}
\end{proof}

To characterize $\parallel_\hipa$ a new parallelism comes in handy.
Let $L_1, L_2\in\lines^\propto$, then

\begin{multline}\label{eq:quadrparal}
  L_1 \quadrparal L_2 \iff
    \text{ there are lines } K_1, K_2, M_1, M_2 \text{ such that }
    \\
    K_1 \astparal K_2 \Land M_1, M_2, L_1 \adjac K_1, K_2 \Land
    L_2 \adjac M_1, M_2 \Land L_1 \not\adjac L_2.
\end{multline}

In plain words: $L_1 \quadrparal L_2$ iff 
$L_1, L_2$ are non-adjacent in $\M\setminus\hipa$, there is a quadrangle $Q$ 
without diagonals such that $L_1$ intersects the lines in one pair of the
opposite sides of $Q$, and $L_2$ intersects the other pair of sides.

\begin{lem}\label{lem:quadrparal}
  Assume that all the lines in $\lines_i$ are of size at least 4 for all $i\in I$
  and $\hipa$ is non-degenerate. If
  $L_1, L_2\in\lines^\propto$ are distinct lines through a point on $\hipa$ and
  there is no $b\in S$, $i\in I$ with\/ $L_1,L_2\subseteq \subst(b,i,S_i)$,
  then  $L_1 \quadrparal L_2$.
\end{lem}

\begin{proof}
  Let $a\in\hipa$ and $a\in L_1\cap L_2$.
  Up to an order of variables we can assume that 
  $L_1 = (l_1,a_2,a_3,\dots)$ and $L_2 =(a_1,l_2,a_3,\dots)$ for some $l_1\in\lines_1$, $l_2\in\lines_2$. 
  For brevity, we omit the coordinates $a_3, a_4,\dots$, which does not affect our reasoning.
  Every point $(a_1,x)$ with $x\neq a_2$ is on $L_2\setminus\hipa$. 
  Let $l_2\ni x\neq a_2$. Then $(l_1,x)$ is a line
  through $(a_1,x)$, which is not contained in $\hipa$.
  If $(y',x)\in (l_1,x)\cap\hipa$ then all points $(y,x)$ on $(l_1,x)$ with $y\neq y'$
  are outside $\hipa$.
  Take $y_1,y_2\in l_1$ such that $y_1,y_2\neq y'$. Clearly the intersection points
  $(y_i,a_2)\in L_1 \cap (y_i,l_2)$ are outside $\hipa$ for $i=1,2$.
  There is exactly one point in $\hipa$ on the line $(y_i,l_2)$, so there is $z\in l_2$
  such that $z\neq a_2,x$ and $(y_i,z)\notin\hipa$.
  The line through the points $(y_1,z)$, $(y_2,z)$ intersects $L_2$ in a point $(a_1,z)$.
  Finally, the points $(y_1,x)$, $(y_1,z)$, $(y_2,x)$, $(y_2,z)$
  give a required quadrangle without diagonals and $L_1,L_2$ both cross their sides, as
  required.
\end{proof}

Two lines in $\M\setminus\hipa$ are parallel if they share a point on $\hipa$, and
there are two possibilities: they arise as lines of one hyperplane complement,
in one variable, or of two distinct hyperplane complements, in two distinct variables.
This observation makes the following fact immediate.

\begin{fact}\label{fct:parallelism}
  Let\/ $\hipa$ be non-degenerate and $L_1, L_2\in\lines^\propto$. 
  Then $L_1\parallel_\hipa L_2$ iff one of the following holds
   \begin{sentences}\itemsep-2pt
  \item\label{inone}
    there is $a\in S$, $i\in I$ such that $L_j=\subst(a,i,l_j)$ for some $l_j\in\lines_i$, $j=1,2$, and
    $l_1\parallel_{\hipa^{[a]}_i} l_2$, or
  \item
    there is no $a\in S$, $i\in I$ as in \eqref{inone} and
	$L_1\quadrparal L_2$.
  \end{sentences}  
\end{fact}

The parallelism $\quadrparal$ can be expressed in terms of the point-line
incidence of $\M\setminus\hipa$ via \eqref{eq:quadrparal} and \ref{lem:quadrparal}. 
To be able to express
$\parallel_\hipa$ in terms of incidence we need to do so with
$\parallel_{\hipa^{[a]}_i}$. The problem is it depends not only on the variable
$i$ but also on $a\in S$. So, we need to distinguish those products $\M$
where every parallelism of any hyperplane complement in $\M_i$
can be defined by a single uniform formula in terms of the point-line incidence 
of that complement for all $i\in I$.
If that is the case \ref{fct:parallelism} is a half-way to express $\parallel_\hipa$
in terms of incidence. What is still missing is an incidence formula for two lines
being in one component of the product.
This is addressed by the next fact which follows from \ref{prop:affofprod:strong} and \cite{naumo}.

\begin{fact}\label{fct:prostewlistku}
  Let\/ $\M_i$ be a Veblenian gamma space with lines of size at least 4 
  and let\/ $\M_i\setminus\hipa^{[a]}_i$ be strongly connected for all $a\in S$, $i\in I$.
  Given two lines $L_1, L_2\in\lines^\propto$
  the following conditions are equivalent:
  \begin{conditions}\itemsep-2pt
  \item
    there is $a\in S$, $i\in I$ such that $L_1,L_2\subseteq\subst(a,i,S_i)$,
  \item
    there is a sequence $Y_0,\dots,Y_m$ of strong subspaces in\/ $\M\setminus\hipa$ 
    such that\/ $L_1\subseteq Y_0$, $L_2\subseteq Y_m$, and $Y_{i-1}, Y_{i}$ share a line for 
    $i=1,\dots,m$.  
  \end{conditions}
\end{fact}

This together with \ref{fct:parallelism} and the remarks below \ref{fct:parallelism} gives

\begin{prop}\label{prop:parallelglobal}
  Let\/ $\M_i$ be a Veblenian gamma space with lines of size at least 4 such that
  the parallelism of every hyperplane complement in\/ $\M_i$
  can be uniformly defined in terms of the point-line incidence of that complement for $i\in I$.
  Assume that $\hipa$ is non-degenerate and\/ $\M_i\setminus\hipa^{[a]}_i$ 
  is strongly connected for all $a\in S$, $i\in I$.
  Then the parallelism\/ $\parallel_\hipa$ in\/ $\M\setminus\hipa$ can be characterized 
  in terms of the point-line incidence of the product\/ $\M$.
\end{prop}

Example~\ref{exm:affinconnected} shows that a strongly connected space could 
turn out to be not connected after affinization. So, the assumption
that hyperplane complements in the components of the product are all strongly connected
is indispensable in \ref{fct:prostewlistku} and \ref{prop:parallelglobal}.

Proposition~\ref{prop:parallelglobal} is 
applicable in case of projective and polar spaces as the parallelism in
question is uniformly definable in affine spaces (a folklore) and in affine
polar spaces (cf.\ \cite{cohenshult}). We do not know however, if parallelism is
uniformly definable in affine Grassmann spaces (we guess so) and in affine polar
Grassmann spaces.


\subsection{Automorphisms}

Whether an automorphism of a hyperplane complement can be extended to  an
automorphism of the ambient space is one of the most common questions  when it
comes to affinization. We have discussed that for partial linear spaces in
Section \ref{sec:auty:pls} and now we are doing so for the Segre product.

\begin{thm}\label{thm:auty:prod}
  Let\/ $\M_i$ be a strongly connected partial linear space for $i\in I$  
  and let $\hipa$ be a \luskwiaty\ hyperplane in\/ $\M=\bigotimes_{i\in I}\M_i$.
  Set\/ $\A := {\M}\setminus\hipa$.
  The automorphisms of\/ $\A$ are the the automorphisms of\/ $\M$
  that preserve $\hipa$, restricted to the point set of\/ $\A$.
  More precisely, a map $f$ is an automorphism of\/ $\A$ iff there is a permutation
  $\sigma$ of\/ $I$ and a family of isomorphisms
  $f_i$ that map $\M_i$ onto
  $\M_{\sigma(i)}$
  such that the map $F$ defined by the condition
  $F(a)_i = f_i(a_i)$
  preserves $\hipa$ and\/ $f$ is the restriction of\/ $F$
  to the point set of\/ $\A$.
\end{thm}

\begin{proof}
 Let $f \in \Aut(\A)$.
 By \ref{prop:wystaje2kolczaty} the hyperplane $\hipa$ is spiky, so according
 to \ref{fct:autyaf:gener}(ii)\eqref{fct:autyaf:gener:a} we have the extension
 $F$ of $f$ to the point set of $\goth N$ such that
 $F(L^\infty) = f(L)^\infty$ for every line $L$ of $\A$.
 As $\hipa$ is \luskwiaty\ in view of \ref{fct:prodpls:gener}\eqref{prodpls:gamma},
 \ref{fct:prodpls:gener}\eqref{prodpls:veblen} and \ref{lem:veblgamma}
 we can apply \ref{fct:autyaf:gener}(ii)\eqref{fct:autyaf:gener:b}. Hence
 $F$ is a collineation of $\goth N$ that preserves $\hipa$.
 Now, from \cite[Proposition~1.10]{naumo} the required $\sigma$ and $f_i$ exist.
\end{proof}

%
\section{Applications and examples: products of partial linear spaces embeddable into projective spaces and their affinizations}\label{sec:apps}

Let $\M := \bigotimes_{i\in I}\M_i=\struct{S,\lines}$ be the Segre product of
partial linear spaces with a hyperplane $\hipa$.
If $\hipa$ is given by \eqref{eq:degenhipa}, then in view of \ref{prop:isomorph}(i) 
the complement $\M\setminus\hipa$ is isomorphic to the product of
hyperplane complements, when both of them are considered as incidence
structures without parallelism. It need not to be true however, in case of other
affinizations. 
From \ref{rem:reduct:notAPLS} we know that if $\hipa$ is spiky, then the
complement $\M\setminus\hipa$ is not an affine space. 

The family of strong subspaces in the hyperplane complement
$\M\setminus\hipa$ will be written as
\begin{equation}\label{eq:sc}
  \covering(\M\setminus\hipa) := \Bset{\subst(a,i,X_i)\colon  a\in S,\ i\in I,\
    X_i \text{ is a strong subspace of } \M_i\setminus\hipa^{[a]}_i}.
\end{equation}
A straightforward outcome of \ref{prop:affofprod:strong} and
\ref{cor:strongasaffine} is as follows
\begin{fact}\label{fact:covering}
  If all\/ $\M_i$ are Veblenian gamma spaces with lines of size at least 4, 
  then $\covering(\M\setminus\hipa)$ 
  is a covering of the hyperplane complement\/ $\M\setminus\hipa$
  by affine spaces, i.e.\ $S\setminus\hipa = \bigcup\covering(\M\setminus\hipa)$.
\end{fact}

As we are interested in affine-like Segre products, due to \ref{fact:covering}
we will investigate products of some analytical Veblenian gamma spaces:
projective spaces, polar spaces, Grassmann spaces,  and polar Grassmann spaces.
All of them are strongly connected.  Thus, \ref{prop:parallelglobal} and
\ref{thm:auty:prod} can be applied as far as there are non-degenerate or
\luskwiaty\ hyperplanes in these spaces. Constructions of hyperplanes with such
properties will be established for products of spaces that are embeddable into
a projective space.
We focus on geometries of common types, although hyperplanes are
also known in  many other embeddable spaces (cf. \cite{ronan}).


\subsection{Algebraic background}

Let $V$ be a (left) vector space over a division ring $D$.
The set of all subspaces of $V$ will be written as $\Sub(V)$
and the set of all $k$-dimensional subspaces as $\Sub_k(V)$. 
For $H\in\Sub_{k-1}(V)$ and $B\in\Sub_{k+1}(V)$ with $H\subseteq B$ a
\emph{$k$-pencil} is the set
$$
  \penc(H,B) := \bset{U\in\Sub_k(V)\colon H\subseteq U\subseteq B}.
$$
Taking $k$-subspaces as points and $k$-pencils as lines we get a \emph{Grassmann space}
$$
  \PencSpace(k,V) := \bstruct{\Sub_k(V), {\cal P}_k(V)}.
$$
For $k=1$, and dually for $k=n-1$ when $V$ is of finite dimension $n$, 
$\PencSpace(k,V)$ is a projective space, while for $1 < k < n - 1$ there are
non-collinear points in $\PencSpace(k,V)$, so it is a proper partial linear space.
It is worth to mention that $\PencSpace(k,V) = \PencSpace(k-1,\PencSpace(1, V))$.

Given a reflexive bilinear form $\xi$ on $V$, we write $Q_k(\xi)$ for the set of all 
isotropic $k$-subspaces of $V$ w.r.t.\ $\xi$.
If $H\in\Sub_{k-1}(V)$, $B\in Q_{k+1}(\xi)$, and $H\subseteq B$ (actually we have $H\in Q_{k-1}(\xi)$), 
then we get an \emph{isotropic $k$-pencil} 
$$
  \penc_\xi(H,B) := \penc(H,B)\cap Q_k(\xi).
$$
Taking isotropic $k$-subspaces as points and isotropic $k$-pencils as lines 
we get a \emph{polar Grassmann space} (cf.\ \cite{polargras})
$$
  \PencSpace(k,\xi) := \bstruct{Q_k(\xi), {\cal G}_k(\xi)}.
$$
It is embedded in the Grassmann space $\PencSpace(k,V)$ in a natural way, so
that the points and lines of $\PencSpace(k,\xi)$ are the points and lines of 
$\PencSpace(k,V)$ respectively. Note that $\PencSpace(1,\xi)$ is a polar space
and $\PencSpace(k,\xi) = \PencSpace(k-1, \PencSpace(1,\xi))$.

Recall that the map
\begin{cmath}
  {\bf g}\colon \gen{u_1,\dots,u_k} \mapsto \gen{u_1\wedge\dots\wedge u_k}
\end{cmath}
provided that $D$ is a field,
is the well known {\em Grassmann embedding} (sometimes called also the
Pl{\"u}cker embedding) of the Grassmann space $\PencSpace(k,V)$ into the projective space 
$\PencSpace(1, {\bigwedge^k V})$.


\subsection{Hyperplanes arising from Segre embeddings}

Let $V_i$ be a vector space over a field $D$ of characteristic not 2 for $i=1,\dots,n$ and let 
$k = k_1+\dots +k_n$ for some positive integers $k_1,\dots, k_n$. 
For brevity of notation we apply a convention that $u^i = [u^i_1,\dots,u^i_{k_i}]\in V_i^{k_i}$
and $u = (u^1,\dots,u^n)$ for $u\in \bigtimes_{i=1}^n V_i^{k_i} =: V$.
Here, we 
investigate the Segre product
\begin{equation}\label{eq:mixprod}
  \M = \M_{k_1,\dots,k_n}(V_1,\dots,V_n) := \PencSpace(k_1,V_1)\otimes\dots\otimes\PencSpace(k_n,V_n).
\end{equation}
Consider a mapping $\mu\colon V\lto D$ that is semilinear and alternating on every 
of $n$ segments w.r.t.\ $k_1,\dots,k_n$, i.e.\ with the property that
\begin{cmath}
  \mu(u_1,\dots,\alpha u_i,\dots,u_k)= \alpha^{\sigma_i}\mu(u_1,\dots, u_i,\dots,u_k)
\end{cmath}
for some automorphism $\sigma_i$ of $D$ and any $\alpha\in D$, and
\begin{cmath}
  \mu(u_1,\dots,u_{j_1},\dots,u_{j_2},\dots,u_k) = - \mu(u_1,\dots,u_{j_2},\dots,u_{j_1},\dots,u_k)
\end{cmath}
for all $i=1,\dots,n$ and $j_1, j_2$ such that $k_1+\dots+k_{i-1}<j_1<j_2\le k_1+\dots+ k_i$. We shall say
that $\mu$ is \emph{segment-wise} semilinear and alternating. Note that $\sigma_{j_1} = \sigma_{j_2}$ 
for $j_1, j_2$ within one segment like above. Thus there could be up to $n$ field automorphisms $\sigma_i$
associated with $\mu$.

For $u\in V$ define a map $\mu_i^{[u]}\colon V_i^{k_i}\lto D$ by setting
$$
  \mu_i^{[u]}(x^i) := \mu(u^1,\dots,u^{i-1},x^i,u^{i+1},\dots,u^n).
$$
It is an alternating $k_i$-semilinear form on $V_i$ associated with some field automorphism $\sigma_i$.
For every map $\mu_i^{[u]}$ there is an alternating $k_i$-linear form $\eta$ on $V_i$ with its zero-set
equal to that of $\mu_i^{[u]}$,
i.e. $\mu_i^{[u]}(x_1,\dots,x_{k_i})=0$ iff $\eta(x_1,\dots,x_{k_i})=0$ for all $x_1,\dots,x_{k_i}\in V_i$.
A $k$-linear form $\mu'$ such that the zero-sets of $\mu$ 
and $\mu'$ coincide exists only if $\sigma_1=\dots=\sigma_n$. This justifies not taking
$\mu$ to be simply $k$-linear.

In case $\sigma_1=\dots=\sigma_n=\id$, that is when $\mu$ is $k$-linear, it determines 
an $n$-linear form $\mu^\ast$ on 
$\bigl(\bigwedge^{k_1} V_1\bigr)\otimes\dots\otimes\bigl(\bigwedge^{k_n} V_n\bigr)$ in a standard way
as follows
\begin{equation}\label{eq:muast}
  \mu^\ast(u^1_1\wedge\dots\wedge u^1_{k_1}\otimes\dots\otimes u^n_1\wedge\dots\wedge u^n_{k_n}) := \mu(u^1,\dots,u^n),
\end{equation}
where $u^i\in V_i^{k_i}$ for $i=1,\dots,n$.

It is known that every hyperplane in the projective space $\PencSpace(1, V)$
is of the form $\Ker(\eta)$ for some linear form or a covector $\eta\in V^\ast$, and indeed
for $n=k=1$ we have $\mu\in V^\ast$.
A standard embedding $\bf s$ of the product $\bigotimes_{i=1}^{n}\PencSpace(1, V_i)$
into the projective space $\PencSpace(1, {\bigotimes_{i=1}^n V_i})$
given by
\begin{equation}\label{eq:segreembed}
  {\bf s}\colon(\gen{w_1},\dots,\gen{w_n}) \longmapsto \gen{w_1\otimes\dots\otimes w_n}.
\end{equation}
is called a \emph{Segre embedding}. 
Let us define
\begin{multline}\label{eq:def:hipamu}
  \hipa_{k_1,\dots,k_n}(\mu) := \bigl\{ \bigl(\gen{u^1},\dots, \gen{u^n}\bigr) 
    \in\Sub_{k_1}(V_1)\times\dots\times\Sub_{k_n}(V_n) \colon \\
      \mu(u^1,\dots,u^n) = 0 \bigr\}.
\end{multline}
Since $n, k, k_1,\dots,k_n$ are all fixed we will abbreviate $\hipa_{k_1,\dots,k_n}(\mu)=\hipa(\mu)$ as it should
cause no confusion.
For $k$-linear $\mu$ we have
  $$\hipa(\mu) = {\bf s}^{-1}({\bf g}_1^{-1}\times\dots\times{\bf g}_n^{-1})(\Ker(\mu^\ast)).$$
For all $u\in V$ such that 
$U:= \bigl(\gen{u^1},\dots,\gen{u^n}\bigr)\in\Sub_{k_1}(V_1)\times\dots\times\Sub_{k_n}(V_n)$ 
by \eqref{eq:hipaslice} and \eqref{eq:def:hipamu} we have
\begin{multline}\label{eq:auxC}
  \hipa\bigl(\mu_i^{[u]}\bigr) = 
    \bset{ \gen{x^i}\in\Sub_{k_i}(V_i)\colon \mu_i^{[u]}(x^i) = 0 } = 
     \\
        \bset{ X\in\Sub_{k_i}(V_i)\colon U[i/X] \in \hipa(\mu) } = 
          \hipa(\mu)_i^{[U]},
\end{multline}
so by \ref{thm:hip:inprod} the following is evident.

\begin{prop}\label{prop:hipamu}
  The set $\hipa(\mu)$ is either a hyperplane in $\M$ or all of\/ $\M$.
\end{prop}

We say that $\mu$ is \emph{non-zero on $i$-th segment} when for all $u\in V$ such that
\begin{equation}\tag{$\ast_i$}\label{eq:ast}
  u^j \text{ is a linearly independent system in } V_j, \text{ where } 1\le j\le n,\; j\neq i
\end{equation}
there is $x^i\in V_i^{k_i}$ with $\mu_i^{[u]}(x^i)\neq 0$.
Note that $u^j$ is linearly independent iff $u^j_1\wedge\dots\wedge u^j_{k_j}\neq 0$ in \eqref{eq:muast}.
Moreover, the system $u^j$ must be linearly independent to have $\gen{u^j}\in\Sub_{k_j}(V_j)$ in \eqref{eq:def:hipamu}.

\begin{prop}\label{fct:mu}
  If the form $\mu$ is non-zero on at least one of $n$ segments, then $\hipa(\mu)$
  is a hyperplane in $\M$. 
  If $\mu$ is non-zero on all $n$ segments, then $\hipa(\mu)$
  is a non-degenerate hyperplane in $\M$.
\end{prop}

Immediately by \ref{fct:mu} we get

\begin{cor}
  There is a (non-degenerate) hyperplane in $\M$.
\end{cor}

Following \cite[Ch.~14.1]{multidim} the form $\mu$ is \emph{GKZ non-degenerate}  
if for all $u\in V$ there is $i\in\set{1,\dots,k}$ and $v\in V$ 
such that $\mu(u_1,\dots,u_{i-1},v_i,u_{i+1}\dots,u_k) \neq 0$.

\begin{rem}
  If $k_1=\dots=k_n=1$, i.e.\ if $\M$ is the Segre product of projective spaces, then 
  the form $\mu$ is GKZ non-degenerate iff $\hipa(\mu)$ is spiky. 
\end{rem}

According to \cite[Ch.~14]{multidim} the form $\mu$ is GKZ non-degenerate iff the hyperdeterminant of
the multidimensional matrix associated with $\mu$ is non-zero. This let us interpret non-zero
hyperdeterminants as those corresponding to spiky hyperplanes in suitable Segre products.

Two papers \cite{hallshult} and \cite{shult} (see also \cite{cuypers}, \cite{debruyn})
provide an exhaustive characterization of hyperplanes in Grassmann spaces. Let us recall
the embeddable case.
\begin{fact}\label{fct:hipygrass}
  Let $n=1$, so $\M$ is a Grassmann space embeddable into a projective space. 
  Then\/ $\hipa$ is a hyperplane in\/ $\M$ iff\/ $\hipa = \hipa(\mu)$ for some non-zero $k$-linear 
  alternating form $\mu$.
\end{fact}

In view of \cite{cohenshult}, by \ref{fct:hipygrass} we get the following
\begin{fact}\label{fct:hipypolar}
  Let $n=k=1$, so\/ $\M$ is a projective space, and let $\xi$ be a bilinear reflexive form on\/ $V$.
  Then\/ $\hipa$ is a hyperplane in the polar space $\PencSpace(1,\xi)$ iff\/
  $\hipa = \hipa(\mu)\cap Q_1(\xi)$ for some non-zero $\mu\in V^\ast$ such that
  $Q_1(\xi)\nsubseteq\hipa(\mu)$.
\end{fact}

Then, as a natural generalization of \ref{fct:hipypolar}, we obtain a formula for hyperplanes in the Segre product of polar Grassmann spaces.

\begin{prop}\label{prop:hip:inpolarprod}
  Let\/ $\xi_i$ be a bilinear reflexive form on\/ $V_i$ for $i=1,\dots,n$.%
  Assume that $\mu$ and $\xi_1,\dots,\xi_n$ satisfy the following condition:
  if\/ $\gen{u^j}\in Q_{k_j}(\xi_j)$ for $j\neq i$, then\/
  $\hipa\bigl(\mu_i^{[u]}\bigr)\cap Q_{k_i}(\xi_i)$
  is neither empty nor a single point for all $u^i\in V_i^{k_i}$, $i=1,\dots, n$. 
  If\/ $\mu$ is non-zero on all $n$ segments and\/ 
  $Q_{k_1}(\xi_1) \times\dots\times Q_{k_n}(\xi_n)\nsubseteq\hipa(\mu)$, then
  $$
    \hipa(\mu,\xi_1,\dots,\xi_n) = 
      \hipa(\mu)\cap\bigl( Q_{k_1}(\xi_1) \times\dots\times Q_{k_n}(\xi_n) \bigr)
  $$
  is a non-degenerate hyperplane in 
  $\PencSpace(k_1,\xi_1)\otimes\dots\otimes\PencSpace(k_n,\xi_n)$.
\end{prop}

\begin{proof}
  Set $\hipa := \hipa(\mu,\xi_1,\dots,\xi_n)$ and take
  \begin{cmath}
    U = \bigl(\gen{u^1},\dots,\gen{u^n}\bigr)\in Q_{k_1}(\xi_1)\times\dots\times Q_{k_n}(\xi_n).
  \end{cmath}  
  Note that $\hipa_i^{[U]} = \hipa\bigl(\mu_i^{[u]}\bigr)\cap Q_{k_i}(\xi_i)$.
  By the assumed condition and \ref{lem:hip:restricted} the set
  $\hipa_i^{[U]}$ is a hyperplane in $\PencSpace(k_i,\xi_i)$
  as the intersection of a hyperplane and the point set $Q_{k_i}(\xi_i)$ of 
  the polar space $\PencSpace(k_i,\xi_i)$ embedded
  into $\PencSpace(k_i, V_i)$ for $i=1,\dots,n$.
  Therefore, $\hipa$ is non-degenerate.
\end{proof}

Some families of non-degenerate hyperplanes were presented so far, but 
in view of \ref{thm:auty:prod} \luskwiaty\ hyperplanes are needed.

We say that $\mu$ is \emph{non-degenerate on $i$-th segment} when
for all $u\in V$ satisfying $(\ast_i)$ 
any linearly independent system $x^i_1,\dots,x^i_{k_i-1}\in V_i$ 
can be completed with $x^i_{k_i}\in V_i$ so that $\mu_i^{[u]}(x^i)\neq 0$.
This notion is a strengthening of a corresponding notion for alternating $k$-linear
forms in \cite{hall}. More precisely, in case $n=1$, i.e. for Grassmann spaces, 
if $\mu$ is non-degenerate, then $\mu$ is non-degenerate in the sense of \cite{hall}, 
while the inverse is true only for $k \le 2$.
Obviously, if $\mu$ is non-degenerate on $i$-th segment, then it is non-zero on $i$-th segment.

\begin{lem}\label{lem:klinear}
  If\/ $\mu$ is non-degenerate on $i$-th segment, then $\hipa\bigl(\mu_i^{[u]}\bigr)$ is
  a \luskwiaty\ hyperplane in $\PencSpace(k_i,V_i)$ for all $u\in V$ satisfying $(\ast_i)$.
\end{lem}

\begin{proof}
  Let us fix $u\in V$ that satisfies $(\ast_i)$ 
  and let $L=\penc(H, B)$ be a line of $\PencSpace(k_i,V_i)$ contained 
  in $\hipa\bigl(\mu_i^{[u]}\bigr)$.
  Assume that $H =\gen{u_1,\dots, u_{k_i-1}}$ for some $u_1,\dots, u_{k_i-1}\in V_i$.
  Note that $B = H\oplus\gen{w_1,w_2}$ for some $w_1, w_2\in V_i$,
  $U_j := H\oplus\gen{w_j}$ are points on $L$, and $\mu_i^{[u]}(u_1,\dots, u_{k_i-1}, w_j) = 0$
  for $j=1,2$.
  As $\mu_i^{[u]}$ is non-degenerate there is $v\in V_i$ such that
  $\mu_i^{[u]}(u_1,\dots, u_{k_i-1}, v) \neq 0$.
  This means that $u_1,\dots, u_{k_i-1}, w_1, w_2, v$ are linearly independent,
  in other words we have a point $U := H\oplus\gen{v}$ in $\PencSpace(k_i,V_i)$ which together
  with $U_1, U_2$ forms a triangle, i.e.\ spans a plane, and $U\notin\hipa\bigl(\mu_i^{[u]}\bigr)$.
\end{proof}

\begin{prop}\label{prop:hipaingrass}
  If\/ $\mu$ is non-degenerate on all $n$ segments, then $\hipa(\mu)$ is a \luskwiaty\ 
  hyperplane in\/ $\M$.
\end{prop}

\begin{proof}
  Note that $\hipa(\mu)$ is a non-degenerate hyperplane.
  Let $U = \bigl(\gen{u^1},\dots,\gen{u^n}\bigr)\in\Sub_{k_1}(V_1)\times\dots\times\Sub_{k_n}(V_n)$. 
  By \ref{lem:klinear} $\hipa(\mu_i^{[u]}) = \hipa(\mu)_i^{[U]}$ is a \luskwiaty\ hyperplane
  in $\PencSpace(k_i,V_i)$ for all $i=1,\dots, n$.
  By \ref{lem:prezerw:wystaw}\eqref{lem:prezerw:wystaw:i} the hyperplane $\hipa(\mu)$
  is \luskwiaty.
\end{proof}

In particular cases,
combining \ref{prop:hip:inpolarprod} and \ref{prop:hipaingrass} yields the formula for 
\luskwiaty\ hyperplanes in the Segre product of polar spaces.

\begin{prop}\label{cor:hip:inpolarprod:1}
  Assume that $k=n$ i.e.\ $k_1 = \dots = k_n = 1$ or\/ $\M$ is the Segre product of projective spaces.
  If\/ $\xi_i$ is a non-degenerate symplectic bilinear form on $V_i$ for $i=1,\dots,n$ 
  and $\mu$ is non-zero (or equivalently non-degenerate in this case) on all $n$ segments, then
  $\hipa(\mu,\xi_1,\dots,\xi_n)$ is a \luskwiaty\ hyperplane in 
  the Segre product $\PencSpace(1,\xi_1)\otimes\dots\otimes\PencSpace(1,\xi_n)$.
\end{prop}

\begin{proof}
  Set $\hipa := \hipa(\mu,\xi_1,\dots,\xi_n)$, take
  $U\in Q_1(\xi_1)\times\dots\times Q_1(\xi_n)$, $i\in I$, and consider the set
  $\hipa_i^{[U]}$.
  Let $L\in Q_2(\xi_i)$ such that $L\subseteq\hipa_i^{[U]}$.
  We take any two points $U_1, U_2$ on $L$ and set 
  $A:= U_1^{\perp_{\xi_i}}\cap U_2^{\perp_{\xi_i}}$. 
  It is impossible that $A\subseteq\hipa_i^{[U]}$. 
  Recall that $Q_1(\xi_i)$ and the point set of $\PencSpace(1, V_i)$ coincide, so
  $\hipa_i^{[U]}$ is not all of $\PencSpace(1, V_i)$.
  We are through by \ref{cor:wystawLS}\eqref{cor:wystawLS:i}.
\end{proof}

When $n=1$ and $k=2$ the form $\mu$ turns out to be a
bilinear symplectic form. Hence, in view of \ref{lem:klinear}, we can state
this.

\begin{cor}\label{cor:radical}
    If\/ $\mu$ is a non-degenerate symplectic bilinear form on  $V$
    (i.e.\ $n=1$, $k=2$), then the set
    $Q_2(\mu)$ of all isotropic  2-subspaces of\/ $V$ w.r.t.\ $\mu$ is a
    \luskwiaty\ hyperplane in\/ $\M_2(V)$. 
\end{cor}

Let us consider the Segre product of two projective spaces and its hyperplanes.

\begin{prop}\label{cor:hip:inproj2prod}
  Let\/ $V_1, V_2$ be vector spaces over a division ring $D$.
  Then the following conditions are equivalent:
  \begin{conditions}\itemsep-2pt
  \item\label{hipinprojprod:war1}
    $\hipa$ is a hyperplane in $\M_{1,1}(V_1,V_2)$.
  \item\label{hipinprojprod:war2}
    There is a sesquilinear form $\xi\colon V_1\times V_2\lto D$
    which determines a conjugacy $\perp$ by the condition that
    $\gen{u_1}\perp\gen{u_2}$ iff\/ $\xi(u_1,u_2)=0$
    for all $u_i \in {V}_i$ and we have $\hipa = \{ (p,q)\colon p \perp q \}$
    (actually $\hipa = \perp$).
  \end{conditions}
\end{prop}

\begin{proof}
  \eqref{hipinprojprod:war1}$\implies$\eqref{hipinprojprod:war2}
  The hyperplane $\hipa$ determines a relation $\perp$ on
  $V_1\times V_2$ by the condition
  $u_1 \perp u_2$ iff $u_1$ is null or $u_2$ is null or
  $\bigl(\gen{u_1},\gen{u_2}\bigr)\in\hipa$.
  In view of \cite[Theorem~32.6]{muzamich} there is a required
  sesquilinear form $\xi$ such that $\perp = \perp_\xi$.

  \eqref{hipinprojprod:war2}$\implies$\eqref{hipinprojprod:war1} Straightforward.
\end{proof}

Note that \ref{cor:hip:inproj2prod} provides examples of hyperplanes corresponding to forms
that are essentially segment-wise semilinear.

\begin{cor}
  There are hyperplanes in $\M_{1,1}(V_1,V_2)$ 
  that do not arise from a Segre embedding. 
\end{cor}

\begin{rem}
  If $k=n>2$ and $\mu$ is alternating, then $\hipa(\mu)$ is non-spiky and 
  thus non-flappy. This, together with \ref{cor:hip:inproj2prod}, means that in the Segre product
  of two projective spaces a hyperplane is flappy iff it is given by some non-degenerate
  2-semilinear form $\mu$ 
  (i.e. $\hipa(\mu)$ is non-degenerate in view of \ref{cor:wystawLS}\eqref{cor:wystawLS:i}), 
  but it is no longer true if the number of factors is more than two.
  
  Moreover, when $k=n>2$ and $D$ is algebraically closed there 
  are no forms $\mu$ that are non-zero (or equivalently non-degenerate in this case) 
  on all $n$ segments. 
  So, it follows that there are no non-degenerate hyperplanes that arise from a
  form (cf. \ref{fct:mu}), and thus there are no flappy hyperplanes by
  \ref{lem:hipynieluskowate}.
\end{rem}

The theory of multilinear forms is definitely complex in general. Even for
3-linear forms there is no complete classification (cf. \cite{shaw}). Therefore,
it should not be expected that the classification of hyperplanes in Segre
products of many factors is possible at the moment.

Clearly, there are non-\luskwiaty\ hyperplanes in $\PencSpace(k,V)$ as well.
An interesting example of such hyperplane can be established without the form $\mu$.
Let $\hipa(W)$ be the set of those $k$-subspaces of $V$ that non-trivially intersect
some fixed subspace $W$ of codimension $k$ in $V$ (cf.\ \cite{cuypers}).
Note that $\hipa(W)$ is a hyperplane in $\PencSpace(k,V)$ regardless of whether it is
embeddable or non-embeddable, while $\hipa(\mu)$ occurs only in embeddable case.

\begin{lem}\label{lem:niekolczaty}
  All hyperplanes of the form\/ $\hipa(W)$ in\/ $\PencSpace(k,V)$ are non-\kolczaty.
\end{lem}

\begin{proof}
  Consider a hyperplane $\hipa = \hipa(W)$ in $\PencSpace(k,V)$.
  Take a point $U_1$ on $\hipa$ such that $\dim(U_1\cap W)\ge 2$.
  Any point $U_2$ not on $\hipa$ is complementary to $W$.
  Suppose that $U_1\adjac U_2\notin\hipa$, i.e.\ $\dim(U_1\cap U_2) = k-1$. Then
  the subspace $U_1\cap U_2$ is a hyperplane in $U_1$ and as such non-trivially
  intersects at least 2-dimensional subspace $U_1\cap W$ of $U_1$. Hence, there
  is a non-zero $w\in U_1\cap U_2\cap W$, a contradiction as $U_2\cap W$ should be trivial.
\end{proof}

Immediately from \ref{prop:wystaje2kolczaty} and \ref{lem:niekolczaty} none of hyperplanes of the form
$\hipa(W)$ is \luskwiaty.
Nevertheless, hyperplanes of this type are used to assemble hyperplanes in 
specific Segre products of Grassmann spaces.

\begin{prop}
  Let $V$ be a a finite-dimensional vector space. 
  For integers $k_1, k_2$ such that\/ $1 < k_1 < \dim(V)-1$ and $k_1 + k_2 = \dim(V)$
  $$
    \hipa_{k_1,k_2}(V):= \bset{ (U_1,U_2)\in\Sub_{k_1}(V)\times\Sub_{k_2}(V) \colon
    0 < \dim(U_1 \cap U_2) }
  $$
  is a non-degenerate non-\kolczaty\ hyperplane in $\M_{k_1,k_2}(V,V)$.
\end{prop}

\begin{proof}
  It suffices to note that for all  
  $U = (U_1,U_2)\in\Sub_{k_1}(V)\times\Sub_{k_2}(V)$ and $i=1,2$
  the set 
  $(\hipa_{k_1,k_2}(V))_i^{[U]} = \hipa(U_{3-i})$
  is, by \ref{lem:niekolczaty}, a non-\kolczaty\ hyperplane in $\PencSpace(k_i,V)$.
  Hence  $\hipa_{k_1,k_2}(V)$ is a hyperplane in our product by \ref{thm:hip:inprod}.
  It is clear that this hyperplane is non-degenerate. 
  So, $\hipa_{k_1,k_2}(V)$ is non-\kolczaty\ by \ref{lem:prezerw:wystaw}\eqref{lem:prezerw:wystaw:ii}.
\end{proof}

The example above is interesting in that the complement
\begin{cmath}
 \M_{k_1,k_2}(V,V)\setminus\hipa_{k_1,k_2}(V)
\end{cmath}
is a pretty well known structure of linear complements i.e.\ the substructure of the respective
product defined on the set 
\begin{cmath}
  \bset{ (U_1,U_2)\in\Sub_{k_1}(V)\times \Sub_{k_2}(V)\colon V=U_1\oplus U_2 }
\end{cmath}
(cf.\ \cite{blunck}, \cite{havlicek}, \cite{pankov}, \cite{segrel}). 
Such structures however are investigated with no parallelism involved in the mentioned papers.


\subsection{Affinization of the product of projective spaces vs the product of affine spaces}

According to \ref{fact:covering} most of affinizations of Segre products of
partial linear spaces are covered by affine spaces. Obviously it does not mean
that these affinizations are, up to an isomorphism, products of affine spaces in
general. Nevertheless, one could suppose that the complement of some hyperplane
in the product of projective spaces will be isomorphic to the product of affine
spaces, as affinizations of all components of this product are exactly affine
spaces. The complement of a degenerate hyperplane given by \eqref{eq:degenhipa}
is, loosely speaking, very close to be that kind of (cf.\ \ref{prop:isomorph}),
and the only barrier is a parallelism.
In view of \ref{lem:prezerw:wystaw}\eqref{lem:prezerw:wystaw:ii} all non-degenerate and a lot of
degenerate hyperplanes in the product of projective spaces are \kolczaty.

Let us stress on that the complements of \kolczaty\ hyperplanes in products of projective
spaces and products of affine spaces are essentially distinct.

\begin{prop}\label{prop:affproj}
  The complement of a \kolczaty\ hyperplane in a product of projective spaces is
  isomorphic to no product of affine spaces.
\end{prop}

\begin{proof}
  The sufficient reason is that according to \ref{rem:reduct:notAPLS} the
  complement in question is not an affine partial linear space, while the product of
  any affine spaces is an affine partial linear space by \ref{fct:prodapls:gener}.
\end{proof}

The Segre product of affine polar spaces, or affine Grassmann spaces, or affine polar Grassmann
spaces seems to be also an affine partial linear space, although it is not straightforward by \ref{fct:prodapls:gener} and may require specific reasoning.
Thus we believe that the following analogy of \ref{prop:affproj} is true.

\begin{conj}
  Let\/ $\A_1$, $\A_2$, $\A_3$ be respectively the complement of \kolczaty\ hyperplane
  in the product of polar spaces, Grassmann spaces, and polar Grassmann spaces. 
  Then
\begin{sentences}\itemsep-2pt
\item
$\A_1$ is isomorphic to no product of affine polar spaces,
\item
$\A_2$ is isomorphic to no product of affine Grassmann spaces,
\item
$\A_3$ is isomorphic to no product of affine polar Grassmann spaces.
\end{sentences}
\end{conj}

We presume that for the product of projective spaces even more can be said.

\begin{conj}\label{conj:prodnotcomp}
   If\/ $\A$ is the complement of a \kolczaty\ hyperplane in a product of
   projective spaces, then the family of lines of\/ $\A$ cannot be completed 
   so that the arising structure is a product of affine spaces.
\end{conj}

To justify \ref{conj:prodnotcomp} let us think through the following example.
\begin{exm}
  Let $V$, $W_i$ be vector spaces over a field $D$ for $i=1,\dots,n$ and let
  $\hipa$ be a \kolczaty\ hyperplane of $\PencSpace(1,V)\otimes\PencSpace(1,V)$.
  Set $\Aproj := \PencSpace(1,V)\otimes\PencSpace(1,V)\setminus\hipa$.
  Assume that $\Aproj$ can be completed by adding some new lines to an affine 
  partial linear space which is the product of affine spaces $\Aaff := \bigotimes_{i=1}^{n}\AfSpace(W_i)$.
  Then the maximal strong subspaces of $\Aproj$ and $\Aaff$ should be isomorphic.
  Moreover
  \begin{cmath}
    \Aut(\Aproj) \cong \Aut(\Aaff).
  \end{cmath}
  Let us compute the size of the corresponding automorphism groups.
  All the calculations will be done under the assumption that 
  $D = \GF(p)$ for a prime $p$, $\dim(V) = m'$, and
  $\dim(W_i) = m$ for all $i=1,\dots,n$.
  Let $\sigma$ be a permutation of $\set{1,\dots,n}$
  Then any collineation of $\Aaff$ is a map $f_{\sigma}=({f_{\sigma}}_1,\dots,{f_{\sigma}}_n)$, where 
  ${f_{\sigma}}_i$ is an isomorphism that maps $\AfSpace(W_i)$ onto
  $\AfSpace(W_{\sigma(i)})$.
  Thus 
  \begin{ctext}
  $\abs{\Aut(\Aaff)} = n!\, p^{m}\abs{\GL(p,m)}$, where
  $\abs{\GL(p,m)}  = (p^{m} - p^0)(p^{m} - p^1)\dots(p^{m}-p^{m-1})$.
  \end{ctext}
  By \ref{cor:hip:inproj2prod} the hyperplane $\hipa$ is determined by some non-degenerate 
  bilinear form $\xi$.
  The automorphisms of $\Aproj$ are exactly the automorphisms of the relation $\perp_\xi$
  (cf.\ \cite{segrel}) and each of these is uniquely determined by a collineation
  $f$ of $\PencSpace(1,V)$ and a permutation in $S_2$. 
  Thus
  \begin{cmath}
  \abs{\Aut(\Aproj)} = 2 \frac{\abs{\GL(p,m')}}{p-1}.
  \end{cmath}
  The sizes of both respective automorphism groups are polynomials in the variable $p$.
  Their degrees and leading coefficients are, respectively,
  ${m}^2+m$,  $n!$, and ${m'}^2-1$, $2$.
  From the isomorphism assumed we get $n=2$, $m' = m$
  and then $m^2 +m = m^2-1$ yields $m=-1$, a contradiction.
\end{exm}



\begin{flushleft}
  \noindent\small
  K. Petelczyc, M. \.Zynel\\
  Institute of Mathematics, University of Bia{\l}ystok,\\
  Akademicka 2, 15-267 Bia{\l}ystok, Poland\\
  \verb+kryzpet@math.uwb.edu.pl+,
  \verb+mariusz@math.uwb.edu.pl+
\end{flushleft}

\end{document}